\documentclass{amsart}
\usepackage{bbm,mathtools}
\usepackage{graphicx}
\usepackage[latin1]{inputenc}
\usepackage{verbatim}

\newtheorem{proposition}{Proposition}
\newtheorem{theorem}{Theorem}
\newtheorem{lemma}{Lemma}

\def\N{{\mathbb N}}
\def\R{{\mathbb R}}
\def\Z{{\mathbb Z}}
\def\P{{\mathbb P}}
\def\E{{\mathbb E}}

\def\cal{\mathcal}
\newcommand\ind[1]{\mathbbm{1}_{\{#1\}}}

\newcommand\alpetu[1]{{\alpha^*_{#1}}}

\def\eps{\varepsilon}
\def\e{\varepsilon}

\def\diff{{\rm{d}}}

\title[]{A Scaling Analysis of a Star Network with Logarithmic Weights}

\author{Philippe Robert}
\address[Ph. Robert]{INRIA Paris \\2 rue Simone Iff, CS 42112, 75589 Paris Cedex 12, France.}
\email{Philippe.Robert@inria.fr}

\author{Amandine V\'eber}
\address[A. V\'eber]{CMAP, École Polytechnique
\\ Route de Saclay, 91128 Palaiseau Cedex, France}
\email{Amandine.Veber@cmap.polytechnique.fr}


\date{\today}
\begin{document}

\begin{abstract}
The paper investigates the properties of a class of resource allocation algorithms for communication networks: if a node of this network has $x$ requests to transmit, then it receives a fraction of the capacity proportional to $\log(1{+}L)$, the logarithm of its current load $L$. A stochastic model of such an algorithm is investigated in the case of the star network, in which $J$ nodes can transmit simultaneously, but interfere with a  central node $0$ in such a way that node $0$ cannot transmit while one of the other nodes does. One studies the impact of the  log policy on these $J+1$ interacting communication nodes.  A  fluid scaling  analysis of the network is derived with the scaling parameter $N$ being  the norm of the initial state. It is shown that the asymptotic fluid behaviour of the system is a consequence of the evolution of the state of the network on a specific time scale $(N^t,\, t{\in}(0,1))$.  The main result is that, on this time scale and under appropriate conditions, the state of a  node with index $j\geq 1$  is of the order of $N^{a_j(t)}$, with $0{\leq}a_j(t){<}1$, where $t\mapsto a_j(t)$ is a piecewise linear function. Convergence results on the fluid time scale  and a stability property are derived as a consequence of this study.
\end{abstract}

\maketitle

\bigskip

\hrule

\vspace{-3mm}

\tableofcontents

\vspace{-1cm}

\hrule

\bigskip

\section{Introduction}
This paper is an extension of the study of algorithms of resource allocation with logarithmic weights started in Robert and V\'eber~\cite{RV}.  For the architectures of communication networks considered, in a wireless context for example,  if two nodes of this network are too close then,  because of interferences, they cannot use the local communication channel at the same time.  For this reason an algorithm has to be designed so that nodes can share the channel in a distributed way in order to transmit their messages. A natural class of algorithms in this setting are random access protocols: A given node waits for  some random duration of time before  transmission. If  the channel is free at that time then it starts  transmitting. Otherwise, if the channel is busy because a communication is already underway in the neighborhood, then the node waits for another random amount of time. For the algorithms investigated in this paper,  the waiting time is exponentially distributed with a rate proportional to the logarithm of the load of the node, i.e., of the form $K\log(1{+}L)$, where $L$ is the number of pending requests at this node and $K$ is some constant. These algorithms are now quite popular, see Shah and Wischik~\cite{Shah}, Bouman et al.~\cite{BBLP} and Ghaderi et al.~\cite{GBW}. They have nice properties in terms of fairness and efficiency. See~\cite{RV} for a discussion of their use in communication networks.

\subsection*{Interaction of Communication Channels}
The results obtained in our previous work~\cite{RV} mainly deal with a network with two nodes. In this case, there is a single communication channel which can be used by only one of the two nodes at any given time. The impact of a $log$-policy was investigated in this case. In the present paper, one considers an additional important feature,  with several communication channels which can be used at the same time provided that they do not interfere.  The network analyzed  has a star topology with $J{+}1$ nodes: there are $J$ nodes, numbered from $1$ to $J$, which can transmit at the same time (i.e., their local communication channels do not interact because of interferences, see Figure~\ref{figstar}), and a central node with index $0$, which is  interacting with  the communication channels of all the other nodes.

As a consequence, node~$0$ cannot transmit at the same time as any of the other nodes. Let $L_i$ be  the current number of pending messages at node $i$. In idle state node~$0$ tries to transmit  at  rate  $K\log(1{+}L_0)$, and the attempt is successful only if all the channels are free, i.e. if none of the nodes with index greater than or equal to $1$ are currently transmitting at that time. When no communication is active, node $0$ is therefore in competition with all the other nodes for transmission. Consequently, it succeeds at rate  $K \log(1{+}L_0)$ or  one of the other nodes starts transmitting at  rate $K(\log(1{+}L_1)+\log(1{+}L_2)+\cdots+\log(1{+}L_J))$.

This situation will  be represented as follows. Suppose the transmission times of requests at node $j$ are exponentially distributed with rate $\mu_j$ and the state of the network of $J{+}1$ queues is $L=(L_j,0\leq j\leq J)$. Then  in our model,  as in~\cite{RV},  any non-empty node with index greater than or equal to $1$ receives the instantaneous capacity $W(L)$ to transmit and node $0$ receives $1-W(L)$ (the total capacity of the channel is assumed to be $1$), where
\begin{equation}\label{WL}
\displaystyle W(L)\stackrel{\text{def.}}{=}\frac{\log(1{+}L_1)+\cdots+\log(1{+}L_J)}{\log(1{+}L_0)+\log(1{+}L_1)+\cdots+\log(1{+}L_J)}.
\end{equation}
In particular node $j$,  $1\leq j\leq J$,  (resp.  node $0$) completes a transmission at rate $\mu_j W(L)$ (resp. $\mu_0(1-W(L))$).  This model assumes in fact that the constant  of proportionality $K$ is sufficiently large so that the waiting times to try to access the channel are negligible.
\begin{figure}[ht] \label{figstar}
\scalebox{0.3}{\includegraphics{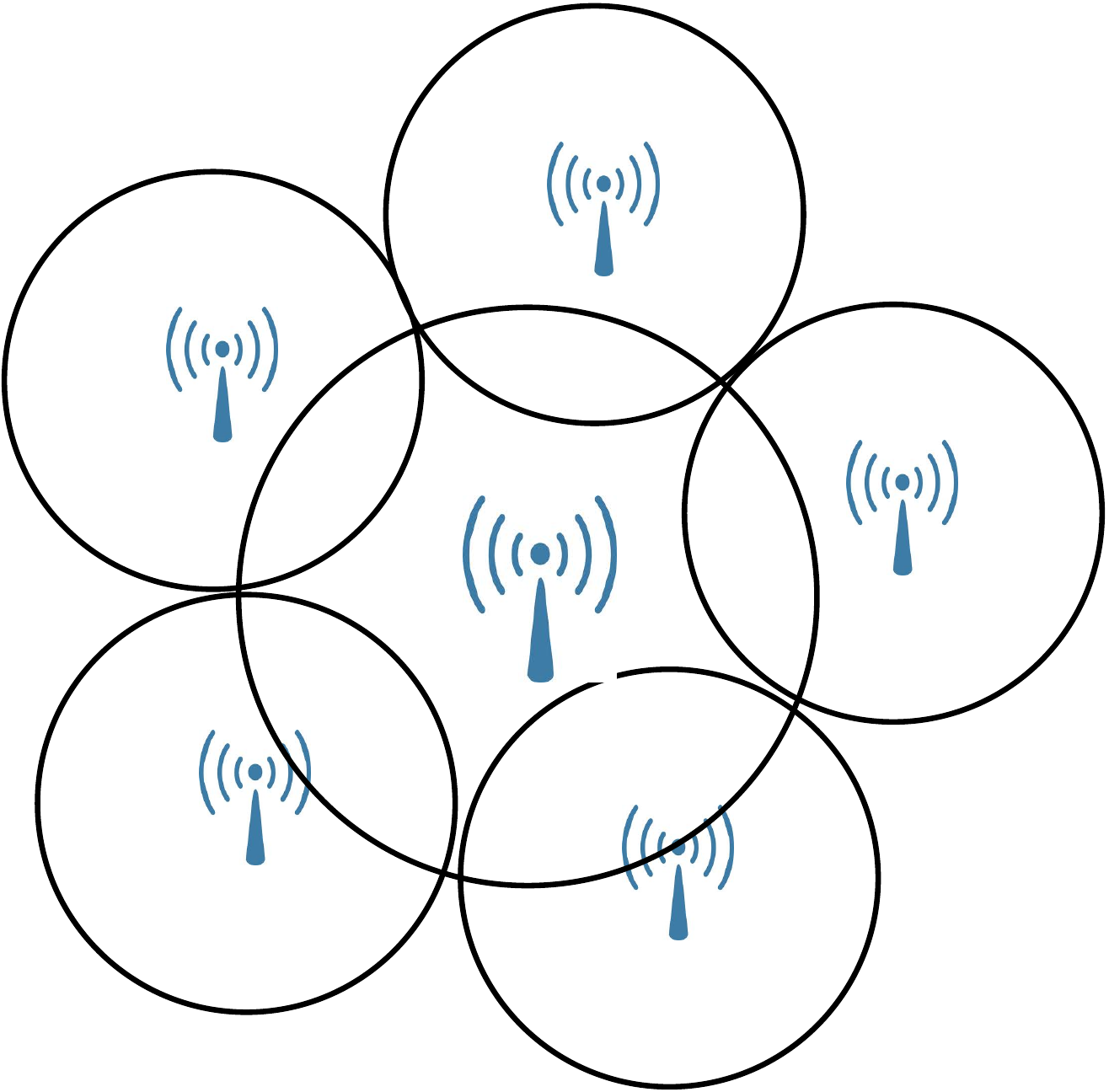}}
\caption{Star Network with $J{=}5$.}
\end{figure}

\subsection*{Assumptions and Notation}
 Requests arrive at node $0\leq j\leq J$ according to a Poisson process with rate $\lambda_j$ and their transmission times  are exponentially distributed with parameter $\mu_j$.  The quantity $\rho_j$ is the load of node $j$, $\rho_j{=}\lambda_j/\mu_j$.  Throughout the paper, without loss of generality one assumes that $\rho_1{<}\rho_2{<}\cdots{<}\rho_J$. That is, excluding node $0$, node $J$ is the most loaded.  One also defines
\begin{equation}\label{alpha}
\alpetu{j}{=}\frac{\rho_j}{1{-}\rho_j},\quad  1\leq j\leq J.
\end{equation}
For $t\geq 0$, $L_j(t)$ denotes the number of requests at node $j$ at time $t$. In what follows, the convergence of a sequence of processes on a time interval $I$  is that associated to the topology of uniform convergence on compact subsets of $I$.

\subsection*{Scaling analysis}
The purpose of this paper is to provide a fluid analysis of this network. This amounts to investigating the convergence properties of the following sequence of processes:
\[
\left( \frac{L_j(Nt)}{N}, 0\leq j\leq J\right),
\]
where $N$ is the norm of the initial state and tends to infinity.
It was shown in Section~7 of~\cite{RV} that such an analysis of the evolution of the state of the network with a fluid scaling  also leads to results on the asymptotic behaviour of the invariant distribution in a heavy traffic regime, and also on the transience properties of the overloaded network.

Among all  possible large initial states, one will consider the most interesting (i.e. difficult) case in which the central node, with index $0$, has $N$ requests and all the other nodes are initially empty:
\begin{equation}\label{initstate}
L_0(0)=N \text{ and } L_j(0)=0 \text{ for } 1\leq j\leq J.
\end{equation}
In the sequel, a superscript $N$ will be used to recall the dependency on $N$ and the process with initial condition~\eqref{initstate} will thus be denoted by
\[
L^N(t)=(L_0^N(t),L_1^N(t), \ldots, L_J^N(t)), \quad t\geq 0.
\]
The other cases for the initial state can be treated in a similar (sometimes easier) way. See the discussion at the end of Section~\ref{SecGen}.  The main problem is to describe how the numbers of requests at the initially empty nodes increase with time and the scaling parameter $N$.  Due to the  interacting communication channels, such an analysis is much more challenging than in our previous work.

To stress the differences with our previous analysis in~\cite{RV}, let us review the main results obtained on the time evolution of the network with two nodes, or $J=1$. In all that follows, one uses the notation $x{\wedge}y=\min \{x,y\}$ and $x^+{=} \max\{x,0\}$.

\subsection{Results for the network with two nodes}\label{ss: two nodes}
This is the case $J{=}1$ with only one communication channel. It was shown in~\cite{RV}  that  two other time scales have to be investigated to understand the convergence properties of the fluid scaling properly.
It turns out that the most interesting case  is when $\rho_1{<}1/2$, or equivalently when $\alpetu{1}$ defined by Equation~\eqref{alpha} satisfies $\alpetu{1}{<}1$.
\begin{enumerate}
\item \label{casa}  {{\em The time scale} ${t\to N^t}$ for $t<\alpetu{1}\wedge 1$.}\ \\
If the initial state is $(L_0^N(0),L_1^N(0))=(N,0)$,  the convergence in distribution
\begin{equation}\label{Nt2}
\lim_{N\to+\infty} \left(\frac{L_1^N(N^t)}{N^t},  0<t<\alpetu{1}\wedge 1\right)= \left(\lambda_1-\mu_1\frac{t}{t+1} ,  0<t<\alpetu{1}\wedge 1\right)
\end{equation}
holds and the first order of the  state  $L_0^N$ of node $0$ stays at $N$ on this time scale.
\medskip
\item {{\em The time scale} ${t\to N^{\alpetu{1}}(\log N)\,t}$.}\ \\
If $(L_0^N(0),L_1^N(0)) =(N,\lfloor N^{\alpetu{1}}\rfloor)$,  the convergence in distribution
\[
\lim_{N\to+\infty}\bigg(\frac{L_1^N(N^{\alpetu{1}}(\log N)\,t)-N^{\alpetu{1}}}{\sqrt{N^{{\alpetu{1}}}\log N}}\bigg)= (Z(t))
\]
holds, where $(Z(t))$ is an Ornstein-Uhlenbeck process.  On this time scale, $L_1^N$ stabilizes around the value $N^{\alpetu{1}}$ and the process $L_0^N$ still remains at $N$.
\medskip
\item {{\em The fluid time scale} ${t\to Nt}$.}\\
The relation
\[
\lim_{N\to+\infty}\bigg(\bigg(\frac{L_0^N(Nt)}{N},\frac{L_1^N(Nt)}{N^{\alpetu{1}}}\bigg),t>0\bigg)=\Big(\big(\gamma(t),\gamma(t)^{\alpetu{1}}\big),t>0\Big)
\]
holds for the convergence in distribution, with
\[
\gamma(t)=(1+(\lambda_0-\mu_0(1-\rho_1))t)^{+}.
\]
\end{enumerate}

\subsection{Evolution of the state of the network with a star topology}
The main results on the fluid behaviour are gathered in Theorem~\ref{th: fluid general2} of Section~\ref{SecGen}. They are summarized as follows.

\medskip
\noindent
    {\bf Convergence on the Fluid Time Scale.} Recall the quantities  $(\alpetu{j})$ defined by Relation~\eqref{alpha} and set
\[
\beta_j^*\stackrel{\text{\rm def.}}{=}\frac{\alpetu{j}}{J{-}j},\, 1{\leq} j{\leq} J, \text{ and }
\kappa\stackrel{\text{\rm def.}}{=}\sup\left\{j\geq 1: \frac{\alpetu{j}}{J{-}j{+}1}<1\right\},
  \]
with $\sup(\emptyset)=0$. The following convergences in distribution hold on a time interval $(0,t_0)$, where $t_0\in (0,\infty]$ depends on the parameters of the network.
  \begin{enumerate}
  \item If $\kappa=0$,
        \[
    \lim_{N\rightarrow \infty}\left(\frac{L_{0}^N(Nt)}{N}, \ldots,\frac{L_{J}^N(Nt)}{N}\right) = (\gamma_{0}(t),\gamma_{1}(t),\ldots,\gamma_J(t)).
    \]
  \item In the case $1\leq \kappa<J$,
    \begin{enumerate}
      \item If $\beta_\kappa^*<1$,
\begin{multline*}
\lim_{N\rightarrow \infty}\left(\frac{L_{0}^N(Nt)}{N}, \frac{L_1^N(Nt)}{(\log N)^3}, \ldots, \frac{L_\kappa^N(Nt)}{(\log N)^3}, \frac{L_{\kappa+1}^N(Nt)}{N}, \ldots,\frac{L_{J}^N(Nt)}{N}\right)
\\= (\gamma_0(t),0^{(\kappa)},\gamma_{\kappa+1}(t),\ldots,\gamma_J(t)).
\end{multline*}
\item If $\beta_\kappa^*>1$,
  \begin{multline*}
\lim_{N\rightarrow \infty}\left(\frac{L_{0}^N(Nt)}{N}, \frac{L_1^N(Nt)}{(\log N)^3}, \ldots, \frac{L_{\kappa-1}^N(Nt)}{(\log N)^3}, \frac{L_{\kappa}^N(Nt)}{N^{\alpha_\kappa^*-(J-\kappa)}}, \frac{L_{\kappa+1}^N(Nt)}{N}, \ldots,\frac{L_{J}^N(Nt)}{N}\right)
\\= \left(\gamma_0(t),0^{(\kappa-1)},\frac{1}{\gamma_{\kappa+1}(t)\gamma_{\kappa+2}(t)\cdots \gamma_J(t)},\gamma_{\kappa+1}(t),\ldots,\gamma_J(t)\right).
\end{multline*}
    \end{enumerate}
\item If $\kappa=J$,
  \[
\lim_{N\rightarrow \infty}\bigg(\frac{L_0^N(Nt)}{N}, \frac{L_1^N(Nt)}{(\log N)^3}, \ldots, \frac{L_{J-1}^N(Nt)}{(\log N)^3} , \frac{L_J^N(Nt)}{N^{\alpetu{J}}} \bigg) = \big(\gamma_0(t) ,0^{(J{-}1)}, \gamma_0(t)^{\alpetu{J}}\big).
\]
  \end{enumerate}
The functions $(\gamma_j(t))$ are deterministic, non-trivial, affine functions. They are defined in Theorem~\ref{th: fluid general2} in Section~\ref{SecGen}. The constant $t_0$ is the first instant on the fluid time scale when the central node empties, i.e. $\gamma_0(t_0){=}0$. Its expression is also given in the statement of the theorem.

The expression given by Relation~\eqref{WL} of the capacity $W(L^N(N\cdot))$ allocated to the nodes with positive indices and the above convergences in distribution show the following property. If $\kappa{>}0$, the states of the  nodes  whose  indices are  between $1$ and $\kappa{-}1$ do not have an impact in the quantity $W(L^N(N\cdot))$, and therefore on  the asymptotic behaviour of the other nodes.  The order of magnitude of the states of these nodes is negligible with respect to any power of $N$ in the fluid regime. As one will see, they behave locally like ergodic $M/M/1$ queues. Hence, on the fluid time scale only a subset of the nodes remain nonnegligible. For example, when $1{<}\kappa{<} J$ and $\beta^*_\kappa>1$, the state of the central node is of the order of $N$, the state of the node with index $\kappa$ is of the order of $N^{\alpetu{\kappa}-(J-\kappa)}$, and all the nodes whose indices lie between $\kappa{+}1$ and $J$ have a number of pending requests which is of the order of $N$.  The other cases exhibit similar behaviours.  Figure~\ref{figstar2} corresponds to the case $\kappa=J$ for $J=3$, but on the time scale $(N^t,t{\in}(0,1))$.

It is interesting to note that these results on the fluid time scale are essentially obtained via a precise analysis of the network on  the time scale $(N^t, t\in(0,1))$.  It should also be noted that  in our previous work~\cite{RV},  the asymptotic analysis of the  behaviour on this time scale, case~\ref{casa}) of Section~\ref{ss: two nodes}, was quite easy in fact. This is not at all the case for the network investigated in the paper.  Some quite technical work has to be done in order to obtain the desired convergence results for the evolution of the processes associated to  the exponents $\log(1{+}L_j^N(N^t))/\log N$,  $1\leq j\leq J$, on this time scale.   Once the asymptotic behaviour of the process $(L_j^N)$ on this time scale is derived, the behaviour of the full network on the fluid time scale $t\mapsto Nt$ can be obtained by using some of the results of~\cite{RV}.

\medskip
{\em Discontinuity on the fluid time scale.}  The convergence results for the fluid scaling are valid on an {\em open } interval $(0,t_0)$ excluding $0$, i.e. on time intervals of the form $[a N, b N]$ with $0{<}a{<}b$, hence ``after'' the time scale $(N^t, t\in(0,1))$. In fact the process exhibits  a kind of discontinuity at $0$ on the fluid time scale, for example in the above case~3) where $\kappa{=}J$. Indeed,   initially $L_J^N(0){=}0$ but $L_J^N(\eps N){\sim}  N^{\alpetu{J}}$ for $\eps{>}0$ arbitrarily close to $0$. In other words, the $J$th coordinate jumps to $N^\alpetu{J}$ at $t=0+$.  This phenomenon can be explained on the time scale  $(N^t, t\in(0,1))$, which is  one of the reasons why this time scale plays a major role in the analysis.

\medskip
{\em Time Varying Exponents in $N$ with Piecewise Affine Behaviours.} From the assumption~\eqref{initstate} on the initial state, only the state of node  $0$ is not zero, and is equal to $N$ initially.  Under appropriate conditions on the $\alpetu{j}$ defined by Equation~\eqref{alpha}, the remarkable feature of the evolution of this network is as follows: on the time scale $(N^t,t\in(0,1))$,  the state of a given node with index $1{\leq}j{\leq}J{-}1$ grows like a power of $N$ until an instant after which it starts {\em decreasing} and finally stabilizes in a finite neighborhood of $0$.  This is in fact  the most difficult technical point of the paper.  See Figure~\ref{figstar2}. Section~\ref{secsec} is essentially devoted to the proof of this result.

Let us describe the phenomenon more precisely.  Recall that the load $\rho_1$ of node~$1$ is such that $\rho_1<\rho_2<\cdots<\rho_J$. If $J\geq 2$ and $\alpetu{1}/J<1$, then
\[
\lim_{N\to+\infty} \left(\frac{L_1^N(N^t)}{N^t}, 0 < t < \frac{\alpetu{1}}{J}\right)= \left(\lambda_1{-}\mu_1\frac{J t}{J t+1}, 0 < t < \frac{\alpetu{1}}{J}\right),
\]
which is a more or less straightforward analogue of Relation~\eqref{Nt2}. More interesting and technically challenging is the behaviour of the process on the ``next'' time interval   $(N^{\alpetu{1}/J}, N^{\alpetu{1}/(J-1)})$ on the time scale $(N^t, t\in(0,1))$: If $\alpetu{1}/(J{-}1)<1$, the convergence in distribution
\begin{equation}\label{eq1}
\lim_{N\to+\infty} \left(\frac{L_1^N(N^t)}{N^{\alpetu{1}-(J-1)t}}, \frac{\alpetu{1}}{J} < t < \frac{\alpetu{1}}{J{-}1}\right){=}\left(\prod_{i=2}^J \frac{1}{\mu_i(\rho_i{-}\rho_1)}, \frac{\alpetu{1}}{J} < t < \frac{\alpetu{1}}{J{-}1}\right)
\end{equation}
holds on this time interval.
This is in fact an equivalent of the equilibrium exponent of the case $J{=}1$ investigated in~\cite{RV}, but it is now time-dependent and stabilizes at $0$ after some time.
Thus, at ``time'' $N^{{\alpetu{1}}/J}$, $L_1^N$ is of the order of $N^{\alpetu{1}/J}$ but just after the exponent in $N$ starts decreasing and is $0$ at ``time'' $N^{{\alpetu{1}}/{(J-1)}}$. Furthermore, at that time the process $L_1^N$ behaves like an ergodic $M/M/1$ queue  and therefore does not scale with any power of $N$ on the time scale $(N^t, t\in(0,1))$, while the states of the other nodes are still of the order of a power of $N$. Consequently, the component $\log(1{+}L_1)$ in the expression~\eqref{WL} of $W(L)$ can be discarded. In other words, after time $N^{{\alpetu{1}}/{(J-1)}}$, the system behaves like a network with node $1$ removed from the architecture. See Figure~\ref{figstar2} for a  representation of the evolution of the queues on the time scale $(N^t,t\in(0,1))$.

By induction, one then shows that  the states of the nodes with indices between $1$ and  $\kappa{-}1$  behave like $\kappa{-}1$ ergodic $M/M/1$ queues. If $\kappa=J$, nodes $0$ and $J$ are the only nodes with a nonnegligible number of requests (with respect to some power of $N$). The consequence is that, in this case, the study of the network is then reduced to the case of a network with two nodes, or $J{=}1$, which is precisely the configuration studied in our previous work~\cite{RV}. Hence, once the behaviour of the star network on the time scale $(N^t, t\in(0,1))$ is understood, its fluid analysis simply follows from the results in~\cite{RV}. In particular, one obtains the result conjectured in Wischik~\cite{Wischik} that under the condition
\[
\rho_0+\max(\rho_j, 1\leq j\leq J) =\rho_0+\rho_J<1,
\]
the Markov process $(L_j(t),0\leq j\leq J)$ is ergodic. It should be noted that Section~9 of~\cite{RV} gives a presentation without proofs of the slightly different (but easier) case of a network with $J$ nodes and the same communication channel, so that only one node can transmit at a time. The techniques which are developed in the present  paper can in fact  be used to establish these results.

\subsection*{Technical difficulties}
 An  important result is an invariance relation, Proposition~\ref{prop: averaging} of Section~\ref{secsec} in the case $J{=}2$ or Theorem~\ref{th: general} of Section~\ref{SecGen}  in the general case. It implies  that the sum of the exponents  $(\log(L_j(N^t))/\log N$, $j\geq 1)$ is constant on specific time intervals. The convergence~\eqref{eq1} is one of the main technical difficulties to establish this result. It turns out to be quite challenging  to show, in particular, that the exponent in $N$   of $(L_1(N^t))$ decreases after time $N^{\alpetu{1}/J}$.   The key technical tool to obtain this convergence is the construction of a new space-time harmonic function~\eqref{approx F} in Section~\ref{secsec} which is used to obtain $L_2$-estimates of the scaled processes $(L_j(N^t))$. When $J\geq 2$ is arbitrary, a family of such space-time harmonic functions is used, see Relation~\eqref{martingale function}.  In spirit, it is connected to some perturbation techniques although it does not seem to be directly related to this framework, see Kurtz~\cite{Kurtz} for example. The convergence~\eqref{eq1} is then  proved in Section~\ref{secsec} using stochastic calculus and several technical estimates  related to the behaviour of reflected random walks.

\setlength{\unitlength}{2144sp}%
\begin{figure}[ht]
\scalebox{0.4}{\includegraphics{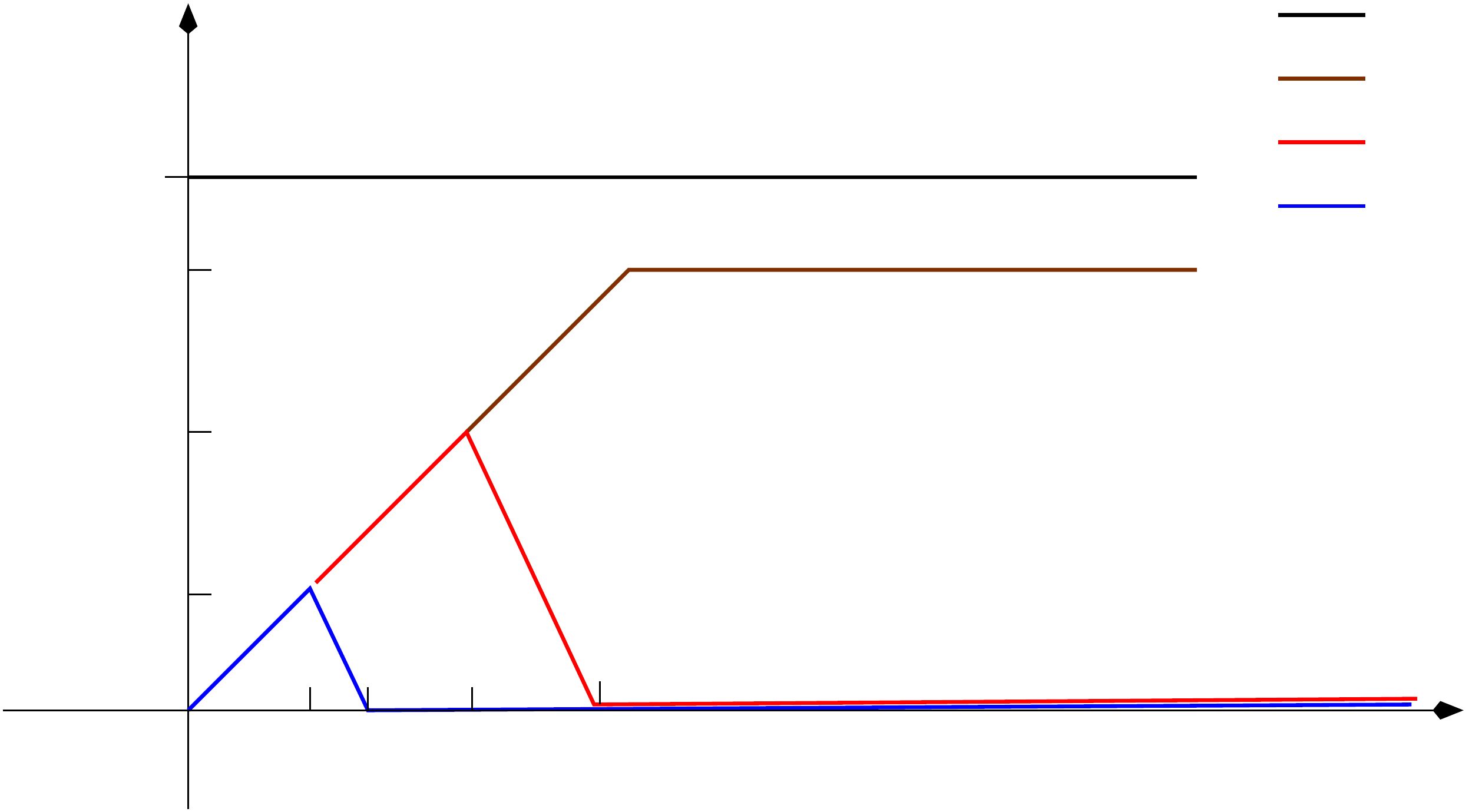}}
\put(-400,4800){$L_0^N$}
\put(-400,4400){$L_3^N$}
\put(-400,4000){$L_2^N$}
\put(-400,3600){$L_1^N$}
\put(-8200,3800){$1$}
\put(-8400,3200){$\alpetu{3}$}
\put(-8400,2200){$\frac{\alpetu{2}}{J{-}1}$}
\put(-8400,1400){$\frac{\alpetu{1}}{J}$}
\put(-7200,200){$\frac{\alpetu{1}}{J}$}
\put(-6800,200){$\frac{\alpetu{1}}{J{-}1}$}
\put(-6200,200){$\frac{\alpetu{2}}{J{-}1}$}
\put(-5400,200){$\frac{\alpetu{2}}{J{-}2}$}
\put(-4400,200){$1$}
\caption{Evolution of $\log(L_j^N(N^t))/\log N$, the exponent in $N$ of $L_j^N$  on the time scale $(N^t, t\in(0,1))$. Here $J{=}3$, $\rho_1{<}\rho_2{<}\rho_3$ and the initial state is $(N,0,0,0)$.}
\label{figstar2}
\end{figure}
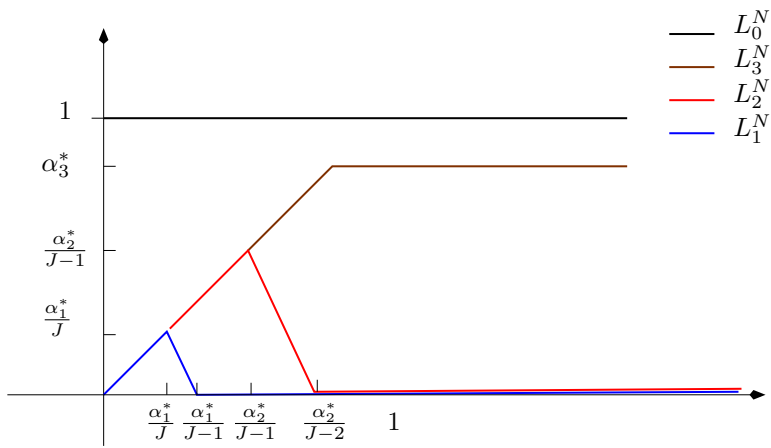

\subsection*{Outline of the paper}
The model and some notation are introduced in Section~\ref{sec: model}, as well as a technical result,  Proposition~\ref{King}, which will be used repeatedly in the subsequent sections. The case  of the network with three nodes, or $J=2$, is studied in detail in Section~\ref{secsec}. It contains the main ingredients to extend the analysis to the general case $J\geq 3$ in Section~\ref{SecGen}.

\section{The Stochastic Model}\label{sec: model}
In this section, one introduces the main stochastic processes and some notation.
For $\xi\geq 0$, ${\cal N}_\xi$ (resp. ${\cal N}_\xi^2$) denotes a Poisson process with rate $\xi$ on $\R_+$ (resp. $\R_+^2$). For any $0 \leq a \leq b$, the quantity ${\cal N}_\xi([a,b])$ denotes the number of points of ${\cal N}_\xi$ in the interval $[a,b]$.  The Poisson processes on $\R_+^2$ are used as follows: If  $x>0$, then ${\cal N}_\xi^2([0,x]\times \cdot )$ is a Poisson process on $\R_+$ with rate $x\xi$. Throughout the paper, the  Poisson processes used are assumed to be independent. If $f$ is an $\R^d$-valued function on $\R_+$, $f(t{-})$ denotes the limit of $f$ to the left of $t>0$, provided that it exists.

The evolution of the Markov process $(L(s))=(L_j(s), 0\leq j\leq J)$ can be described as the solution to the following stochastic differential equation (SDE)
\begin{equation}\label{SDE1}
\begin{cases}
\diff L_0(s)&=\displaystyle {\cal N}_{\lambda_0}(\diff s) -\ind{L_0(s-)>0} {\cal N}_{\mu_0 }^2([W(L(s-)),1],\diff s),\\ \ \\
\diff L_j(s)&= \displaystyle{\cal N}_{\lambda_j}(\diff s) -\ind{L_j(s-)>0} {\cal N}_{\mu_j }^2([0,W(L(s-))), \diff s),\quad 1\leq j\leq J,
\end{cases}
\end{equation}
where $L(s)=(L_0(s),L_1(s),\ldots, L_J(s))$ and $W(\cdot)$ is defined by Equation~\eqref{WL}.

For $0\leq j\leq J$, one denotes the solution to the SDE~\eqref{SDE1} with initial state $(N,0,\ldots,0)$ by $(L_j^N(s))$, and $Y_j^N(s)$ describes the exponent in $N$ of $L_j^N(s)$:
\begin{equation}\label{def Y}
Y_j^N(s)=\frac{\log(1+L_j^N(s))}{\log N}.
\end{equation}
Recall that $\rho_j=\lambda_j/\mu_j$. Without loss of generality, one assumes that nodes $1, \ldots, J$ are ordered so that $\rho_1<\rho_2<\cdots<\rho_J$.

The following result is a simple consequence of a result of Kingman~\cite{Kingman} in the case of birth and death processes. It will be used repeatedly.
\begin{proposition}\label{King}
\begin{enumerate}
\item[]
\item If $(X(s))$ is a birth and death process on $\Z$ starting at $1$ with birth rate $\lambda$ and death rate $\mu>\lambda$, then for any integer $x\geq 0$,
\begin{equation}\label{KingEq}
\P\left(\sup_{s\geq 0} X(s)\geq x \right) \leq \left(\frac{\lambda}{\mu}\right)^x.
\end{equation}
\item  If $(X_+(s))$ denotes the process with the same transitions as $(X(s))$ but with a reflection at $0$, then for any $T>0$,
\begin{equation*}\label{KingEq2a}
\P\left(\sup_{0\leq s\leq T} X_+(s)\geq x \right) \leq(\lambda T+1) \left(\frac{\lambda}{\mu}\right)^x
\end{equation*}
and
\begin{equation*}\label{KingEq2b}
\E\left(\sup_{0\leq s\leq T}X_+(s)^2\right)\leq 2(\lambda T+1)\frac{\mu^2}{(\mu-\lambda)^2}.
\end{equation*}
\end{enumerate}
\end{proposition}
\begin{proof}
Relation~\eqref{KingEq} is  Relation~(3.3) in Theorem 3.5 of Robert~\cite{Robert} for example. The second and third relations follow by remarking that the sample paths of $(X_+(s))$ can be obtained as a concatenation of excursions of $(X(s))$ above $0$. The estimate~\eqref{KingEq}  is in fact an upper bound on the probability that the supremum of an excursion is greater than $x$. Two excursions are separated by at least an exponential random variable with parameter $\lambda$, so that the total number of such excursions in the interval is stochastically bounded by 1 plus a Poisson random variable with parameter $\lambda T$. The proposition is proved.
\end{proof}
Let us begin with the case of a network with three nodes, or $J=2$. The general case is analysed in Section~\ref{SecGen}. As before, in all that follows, the convergence of processes is that associated to the topology of uniform convergence on compact subsets of the time interval of interest.

\section{Three node network}\label{secsec}
In this section one assumes that $J=2$, $\rho_1<\rho_2$. Recall that the initial state is given by  $L_0^N(0)=N$ and $L_1^N(0)=0=L_2^N(0)$.  The main results of this section can be summarized briefly as follows.
\begin{enumerate}
\item On the time interval $(0,(\alpetu{1}/2)\wedge 1)$, $L^N_1(N^t)$ and $L^N_2(N^t)$ grow like $C(t)N^t$, for some linear functions $C_1(t)$ and $C_2(t)$.
\item If $\alpetu{1}/2<1$, on the time  interval $(\alpetu{1}/2,\alpetu{1}\wedge 1)$,
  \begin{enumerate}
    \item $(L^N_1(N^t))$ decreases like $c_1N^{\alpetu{1}-t}$ for some constant $c_1>0$. If $\alpetu{1}<1$, then it reaches a neighbourhood of $0$ in which it remains for the rest of the evolution.
    \item  $(L^N_2(N^t))$ still grows like $c_2N^t$ for some constant $c_2$ such that  $c_1c_2=1$. If $\alpetu{2}<1$, then it remains in a close neighbourhood of $N^{\alpetu{2}}$ until $t=1$ when the fluid time scale ``begins''.
  \end{enumerate}
\end{enumerate}
First, the following proposition gives the behaviour of the network up to time $N^{\alpetu{1}/2}$. Its proof is identical to that of Proposition~2 of Robert and V\'eber~\cite{RV}, and is therefore omitted. 

\begin{proposition}\label{prop: phase 1}
The convergence of processes
\[
\lim_{N\rightarrow \infty} \bigg(\frac{L^N_1(N^t)}{N^t},\frac{L^N_2(N^t)}{N^t}\bigg)  = \bigg(\lambda_1 - \mu_1\frac{2t}{1+2t},\, \lambda_2 - \mu_2\frac{2t}{1+2t}\bigg)
\]
holds on the time interval  $(0,{\alpetu{1}}/{2}\wedge 1)$.
\end{proposition}
If $\alpetu{1}/2 > 1$, for a suitably small $\eps>0$ then one has that $L_i^N(\eps N)$ is of the order of $N$ for every $i\in \{0,1,2\}$, and the fluid analysis of this network is straightforward. See \ref{th2a}) of Theorem~\ref{th: fluid limit} below.

The first theorem describes the network on the time scale $t\mapsto N^t$ on the time interval $(\alpetu{1}/2,\alpetu{1}\wedge 1)$.
\begin{theorem}\label{th: averaging} Under the assumption that $\alpetu{1}/2<1$, the convergence
\[
\lim_{N\rightarrow \infty} \bigg(\frac{L_1^N(N^t)}{N^{\alpetu{1}-t}},\frac{L_2^N(N^t)}{N^t}\bigg) = \bigg(\frac{1}{\mu_2(\rho_2-\rho_1)}, \mu_2(\rho_2-\rho_1)\bigg)
\]
holds on the time interval $(\alpetu{1}/2,\alpetu{1}\wedge 1)$.
\end{theorem}
In particular, $(L_1^NL_2^N(N^t)/N^{\alpetu{1}})$ converges to the process constant equal to one on this time interval.

When $\alpetu{1}<1$ one shows that after ``time'' $N^{\alpetu{1}}$, $L_1^N$ remains of order $(\log N)^2$ at most, so that the system $(L_0^N,L_2^N)$ is indeed equivalent to the 2-queue system analyzed in \cite{RV}.
\begin{proposition}\label{prop: phase 3}
  Under the condition $\alpetu{1}<1$, if
  \[
  \theta_0^N=\inf\{t>0: L_1^N(t)=0\}
  \]
  then
\begin{enumerate}
\item for $\eps>0$, one has
\[
\lim_{N\rightarrow \infty} \P\big(\theta_0^N \leq N^{\alpetu{1}}\log N + N^\eps \log N\big)= 1.
\]
\item For  $\eps>0$ sufficiently small, one has
\[
\lim_{N\rightarrow \infty} \P\Bigg(\sup_{t\in [\theta_0^N,N^{(\alpetu{2}\wedge 1)-\eps}]} L_1^N(t)  \leq (\log N)^2\Bigg) = 1.
\]
\item The convergence of processes
\[
\lim_{N\rightarrow \infty}\bigg(\frac{L_2^N(N^t)}{N^t}\bigg) = \bigg(\lambda_2 - \mu_2\,\frac{t}{1+t}\bigg)
\]
holds on the time interval $(\alpetu{1},\alpetu{2}\wedge 1)$.
\end{enumerate}
\end{proposition}
Finally, the second theorem gives the fluid limit of the network with three nodes.
\begin{theorem}[Fluid Limits]\label{th: fluid limit} The following convergences of processes hold on the time  interval $(0,t_0)$. 
\begin{enumerate}
\item \label{th2a} If $\alpetu{1}/2>1$, then $t_0={3}/(\mu_0{-}3\lambda_0)^+$  and 
\begin{multline*}
\lim_{N\rightarrow \infty} \bigg(\frac{L_0^N(Nt)}{N}, \frac{L_1^N(Nt)}{N}, \frac{L_2^N(Nt)}{N}\bigg) \\
 = \bigg( 1{+}\mu_0\Big(\rho_0{-}\frac{1}{3}\Big)t, \mu_1\Big(\rho_1{-}\frac{2}{3}\Big)t, \mu_2\Big(\rho_2{-}\frac{2}{3}\Big)t \bigg).
\end{multline*}
\item  \label{th2b} If $\alpetu{1}/2<1<\alpetu{1}$, then $t_0 = 1/(1{-}\rho_0{-}\rho_1)^+ $ and 
\begin{multline*}
\lim_{N\rightarrow \infty} \bigg(\frac{L_0^N(Nt)}{N}, \frac{L_1^N(Nt)}{N^{\alpetu{1}-1}}, \frac{L_2^N(Nt)}{N} \bigg)\\
  = \bigg( 1 {+}\mu_0(\rho_0 {+}\rho_1{-}1)t, \frac{1}{\mu_2(\rho_2{-}\rho_1)t}, \mu_2(\rho_2 {-}\rho_1)t\bigg).
\end{multline*}
\item  \label{th2c} If $\alpetu{1}<1<\alpetu{2}$, then $t_0=1/(\mu_0(1/2-\rho_0))^{+}$ and 
\[
\lim_{N\rightarrow \infty}\bigg(\frac{L_0^N(Nt)}{N}, \frac{L_1^N(Nt)}{(\log N)^3}, \frac{L_2^N(Nt)}{N} \bigg) 
 = \bigg(1{+}\mu_0\Big(\rho_0{-}\frac{1}{2}\Big)t,0, \mu_2\Big(\rho_2{-}\frac{1}{2}\Big)t\bigg).
\]
\item  \label{th2d} If $\alpetu{2}<1$, then $t_0=+\infty$ and 
\begin{align*}
\lim_{N\rightarrow \infty}&\bigg(\frac{L_0^N(Nt)}{N}, \frac{L_1^N(Nt)}{(\log N)^3}, \frac{L_2^N(Nt)}{N^{\alpetu{2}}} \bigg) = \big((\gamma(t), 0, \gamma(t)^{\alpetu{2}})\big),
\end{align*}
with $\gamma(t)= (1{+}\mu_0(\rho_0 {+}\rho_2 {-} 1)t)^+$. 
\end{enumerate}
\end{theorem}
The rest of this section is devoted to the proofs of these results. Theorem~\ref{th: averaging} is proved in Section~\ref{ss:proof 1}, Proposition~\ref{prop: phase 3} in Section~\ref{ss:proof 2} and Theorem~\ref{th: fluid limit} in Section~\ref{ss:proof 3}.

\subsection{Proof of Theorem~\ref{th: averaging}}\label{ss:proof 1}
This is done in several steps. Since $L_0^N(0)=N$ and one is concerned with the time scale $(N^t,\, 0<t<1)$, the fluctuations of the process $(L_0^N(t))$  around $N$ are negligible for our purpose. It will be implicitly assumed that $L_0^N\equiv N$ on this time scale. To make this rigourous, one can proceed as in the proofs of related results in~\cite{RV}, see the proof of Proposition~1 in this reference for example,  and use a coupling of $L_0^N$ with its arrival process and with its departure process to establish that the results below hold in these worst-case scenarios.

First, from Proposition~\ref{prop: phase 1} one sees that at time $N^{\alpetu{1}/2}$, the drift term of node $1$ cancels while that of node $2$ is still positive. This suggests that, at least for a small amount of time after $N^{\alpetu{1}/2}$, $(L^N_2(t))$ keeps on increasing. The following lemma establishes a preliminary result in this direction.
\begin{lemma}\label{lem: first bounds}
For $\gamma\in \big({\alpetu{1}}/{2},\alpetu{2}/2\wedge 1\big)$,  the relation
\begin{equation*}\label{growth L2}
\lim_{N\rightarrow \infty}\P\left(\frac{1}{2}\left(\lambda_2{-}\mu_2\frac{2\gamma}{1{+}2\gamma}\right) <\hspace{-3mm} \inf_{s\in [N^{\alpetu{1}/2},N^\gamma]}\hspace{-3mm} \frac{L^N_2(s)}{s}\leq \hspace{-3mm} \sup_{s\in [N^{\alpetu{1}/2},N^\gamma]}\hspace{-3mm} \frac{L^N_2(s)}{s} {<} 2\lambda_2\right) {=} 1,
\end{equation*}
holds and there exists some  $A_\gamma>0$ such that
\begin{equation*}\label{growth L1}
\lim_{N\rightarrow \infty} \P\left(\sup_{s\in [N^{\alpetu{1}/2},N^\gamma]}L^N_1(s) {<} A_\gamma N^{\alpetu{1}/2}\right)= 1.
\end{equation*}
\end{lemma}

\begin{proof} The processes $(L^N_1(s))$ and $(L^N_2(s))$ are stochastically bounded from above by their arrival processes, which are Poisson processes with respective rates $\lambda_1$ and $\lambda_2$. Hence, the ergodic theorem for Poisson processes gives
\begin{equation}\label{crude estimate}
\lim_{N\rightarrow \infty}\P\left(\sup_{s\in [N^{\alpetu{1}/4}, N^\gamma]}\frac{L^N_2(s)}{s} {<} 2\lambda_2\right){=} \lim_{N\rightarrow \infty}\P\left(\sup_{s\in[N^{\alpetu{1}/4}, N^\gamma]}\frac{L^N_1(s)}{s} {<} 2\lambda_1\right){=}1.
\end{equation}
Recall the notation $Y_j^N$ introduced in (\ref{def Y}). Thus, for every $s\in [N^{\alpetu{1}/4}, N^{\gamma}]$ and some appropriate $C>0$, one has
\begin{multline*}
\frac{Y_1^N(s) + Y_2^N(s)}{1+ Y_1^N(s) + Y_2^N(s)}\leq \frac{\log (2\lambda_1 N^\gamma) + \log(2\lambda_2 N^\gamma)}{\log N + \log (2\lambda_1 N^\gamma) + \log(2\lambda_2 N^\gamma)} \leq  \frac{2\gamma}{1+2\gamma} + \frac{C}{\log N}
\end{multline*}
with probability tending to $1$. As a consequence, on the time interval $[N^{\alpetu{1}/4},N^\gamma]$ the process $(L^N_2(s))$ is stochastically bounded from below by the process
\[
\Big({\cal N}_{\lambda_2}[N^{\alpetu{1}/4},s] - \mathcal{N}_{ {2\gamma\mu_2}/{(1+2\gamma)} + {C\mu_2}/{\log N}}[N^{\alpetu{1}/4},s]\Big).
\]
Since $\gamma < {\rho_2}/{(2(1-\rho_2))}$ by assumption, one has $\lambda_2> 2\mu_2\gamma/{(1+2\gamma)}$, which enables one to conclude that
\begin{equation}\label{behaviour L2}
\lim_{N\rightarrow \infty}\P\bigg(\inf_{s\in [N^{\alpetu{1}/2},N^\gamma]}\frac{L^N_2(s)}{s} > \frac{1}{2}\, \bigg(\lambda_2 - \mu_2 \frac{2\gamma}{1+2\gamma}\bigg)\bigg) = 1.
\end{equation}
Together with (\ref{crude estimate}), this proves the first statement of the lemma. Furthermore, the event
\[
{\cal E}_N\stackrel{\text{def.}}{=}\bigg\{\sup_{s\in [N^{\alpetu{1}/4},N^\gamma]}\frac{L^N_1(s)}{s} {<} 2\lambda_1\bigg\} \bigcap \bigg\{\inf_{s\in [N^{\alpetu{1}/2},N^\gamma]}\frac{L^N_2(s)}{s} {>} \frac{1}{2}\, \bigg(\lambda_2 {-} \mu_2 \frac{2\gamma}{1{+}2\gamma}\bigg)\bigg\}
\]
has a probability arbitrarily close to $1$ for $N$ large enough.

Set $\eta =\big(\lambda_2 - 2{\gamma\mu_2}/{(1+2\gamma)}\big)/2$ and fix $A>1$  such that $A\eta>1$.  On the event ${\cal E_N}$, if at some instant $t_0\geq AN^{\alpetu{1}/2}$ one has $L^N_1(t_0)\geq N^{\alpetu{1}/2}$, then the total service rate of class $1$ jobs at that time is bounded from below by
\begin{multline*}
\mu_1 \, \frac{\log(N^{\alpetu{1}/2}) + \log(\eta A N^{\alpetu{1}/2})}{\log N + \log(N^{\alpetu{1}/2}) + \log(A\eta N^{\alpetu{1}/2})} \\
= \mu_1 \frac{\alpetu{1}}{1+\alpetu{1}} + \mu_1\, \frac{\log(\eta A)}{(\alpetu{1}+1)\log N}\frac{\log{N}}{\log(\eta A)+\log(N)}
\geq  \lambda_1 + \frac{C}{\log N},
\end{multline*}
for all $N\geq 2$ and some constant $C$. Since $A\eta>1$, one can choose $C>0$ here. Hence, on  the time interval $[A N^{\alpetu{1}/2}, N^\gamma]$, when  the process  $(L^N_1(s))$ is above the level $N^{\alpetu{1}/2}$, it is stochastically bounded (from above) by the process $N^{\alpetu{1}/2} + X_+(s)$, where $(X_+(s))$ is  a birth and death process reflected at $0$ with birth rate $\lambda_1$ and death rate $\beta_1{=}  \lambda_1 + {C}/{\log N}$. From Proposition~\ref{King}b) one obtains
\begin{equation}\label{excursion}
\P\left(\sup_{s\in[A N^{\alpetu{1}/2},N^\gamma]} X_+(s)\geq N^{\alpetu{1}/2}\right)\leq (\lambda_1 N^\gamma+1)\left( \frac{\lambda_1}{\lambda_1+C/\log N}\right)^{\mathclap{N^{\alpetu{1}/2}}}.
\end{equation}
In particular, on the event ${\cal E}_N$ none of the excursions of $(L^N_1(s))$ above $N^{\alpetu{1}/2}$ will exceed the value $2N^{\alpetu{1}/2}$ with a probability bounded by the quantity in the right hand side of (\ref{excursion}). This yields
\begin{multline*}
\limsup_{N\to+\infty}\P\left(\sup_{s\in [A N^{\alpetu{1}/2},N^\gamma]} L^N_1(s)\geq 2N^{\alpetu{1}/2} \right)\\
=\limsup_{N\to+\infty} \P\left(\left\{\sup_{s\in [A N^{\alpetu{1}/2},N^\gamma]}L^N_1(s)\geq 2N^{\alpetu{1}/2}\right\}\cap{\cal E}_N \right)\\
 \leq \limsup_{N\to+\infty} \lambda_1 N^\gamma\left( \frac{\lambda_1}{\lambda_1+C/\log N}\right)^{\mathclap{\qquad N^{\alpetu{1}/2}}}=0.
\end{multline*}
There remains to control $(L^N_1(s))$ on the time interval $[N^{\alpetu{1}/2},A N^{\alpetu{1}/2}]$. This is done with the help of Relation~\eqref{crude estimate}, which gives the identity
\[
\lim_{N\to+\infty} \P\left(\sup_{s\in [N^{\alpetu{1}/2},A N^{\alpetu{1}/2}]} L^N_1(s) < 2\lambda_1 A N^{\alpetu{1}/2} \right)=1.
\]
Taking $A_\gamma= 2+2\lambda_1 A$, the lemma is proved.
\end{proof}
Theorem~\ref{th: averaging} states in particular that the sequence of processes $(L^N_1(s)L^N_2(s)/N^{\alpetu{1}})$ is  converging to $(1)$ on the time interval $({\alpetu{1}}/{2},\alpetu{2}/2{\wedge}\alpetu{1}{\wedge}1)$.  The following result  gives a weaker version of that. It shows that it is true for the log scale, i.e. for $(Y_1^N(s))$ and $(Y_2^N(s))$, the exponents in $N$ of $(L^N_1(s))$ and $(L^N_2(s))$.
\begin{lemma}\label{lem: uniform estimate}
For every $\gamma\in \big({\alpetu{1}}/{2},\alpetu{2}/2 \wedge \alpetu{1}\wedge 1\big)$ and every $\eps>0$, one has
\[
\lim_{N\rightarrow \infty} \P\bigg(\sup_{t\in [\alpetu{1}/2,\gamma]}\big|Y_1^N(N^t)+Y_2^N(N^t)-\alpetu{1}\big| > \eps\bigg)= 0.
\]
\end{lemma}
\begin{proof}
Let us first consider the initial time $N^{\alpetu{1}/2}$. By Lemma~\ref{lem: first bounds}, there exist constants $C_1$, $C_2$ and $C_2'>0$ such that
\[
\lim_{N\rightarrow \infty} \P\big(L^N_1(N^{\alpetu{1}/2})\leq C_1 N^{\alpetu{1}/2} \hbox{ and }C_2' N^{\alpetu{1}/2} \leq L^N_2(N^{\alpetu{1}/2}) \leq C_2 N^{\alpetu{1}/2} \big) = 1.
\]
Let us first complete this result by showing that for every $0<\eps <\alpetu{1}/4$, one has
\begin{equation}\label{initial L1}
\lim_{N\rightarrow \infty} \P\big(L^N_1(N^{\alpetu{1}/2}) \geq C_1' N^{\alpetu{1}/2 -\eps}\big) = 1
\end{equation}
for some constant $C_1'=C_1'(\eps)>0$.
By Proposition~\ref{prop: phase 1}, there exists $C=C(\eps)>0$ such that
\[
\lim_{N\to+\infty}\P\big(L^N_1(N^{\alpetu{1}/2 - \eps})\geq C N^{\alpetu{1}/2 - \eps}\big)=1.
\]
Another use of Proposition~\ref{prop: phase 1} and Lemma~\ref{lem: first bounds} shows that the process  $L_1^N$ have  transitions such that
\[
\begin{cases}
x\rightarrow x+1\mbox{ at rate }\lambda_1\\
\displaystyle x \rightarrow x-1\mbox{ at a rate} \leq \mu_1 \frac{\alpetu{1}}{(1+\alpetu{1})} = \lambda_1
\end{cases}
\]
 on the time interval $[N^{\alpetu{1}/2-\eps},N^{\alpetu{1}/2}]$. 
Consequently,  the process $(L^N_1(s+N^{\alpetu{1}/2-\eps}))$ is stochastically bounded from below by $(C N^{\alpetu{1}/2 - \eps} + X(s))$  on the time interval $[0,N^{\alpetu{1}/2}{-} N^{\alpetu{1}/2 - \eps}]$, where $X$ denotes a symmetric random walk jumping up and down by $1$ at rate $\lambda_1$ in each direction. Writing $I_N=[0,N^{\alpetu{1}/2}-N^{\alpetu{1}/2 -\eps}]$, Doob's Inequality applied to the  martingale $(X(s))$ shows that for any $\kappa>0$,
\begin{multline*}
  \P\Big(\inf_{s\in I_N}L^N_1(s+N^{\alpetu{1}/2-\eps}) \leq  C N^{\alpetu{1}/2 - \eps} {-} N^{\kappa} \Big)\\ \leq \P\Big(\inf_{s\in I_N}X(s) \leq {-}N^{\kappa} \Big) 
\leq \frac{2\lambda_1}{N^{2\kappa}}N^{\alpetu{1}/2 },
\end{multline*}
by Proposition~\ref{King}.
Since $\eps<\alpetu{1}/4$, $\kappa$ can be chosen so that $\alpetu{1}/4<\kappa<\alpetu{1}/2- \eps$, and then
\[
\lim_{N\to+\infty} \P\Big(\inf_{s\in I_N}L^N_1(s) \leq  C N^{\alpetu{1}/2 - \eps} - N^{\kappa} \Big)=0.
\]
Relation~\eqref{initial L1} follows.

The next step is to show that
\[
\lim_{N\rightarrow \infty} \P\bigg(\inf_{s\in [N^{\alpetu{1}/2},N^{\gamma}]}\left(Y_1^N(s)+Y_2^N(s)\right)< \alpetu{1} -\eps\bigg) = 0.
\]
From Lemma~\ref{lem: first bounds}, one has that
\[
Y_2^N(N^t)= t +\mathcal{O}(1/\log N),\quad  \forall t\in \left[{\alpetu{1}/2},{\gamma}\right]
\]
holds  with a probability tending to $1$ as $N$ tends to infinity. Hence, all one has to prove is that
\begin{equation}\label{eqauxS}
\lim_{N\to+\infty}\P\left(\inf_{t\in[\alpetu{1}/2,\gamma]}\left(Y_1^N(N^t)- \alpetu{1} +\eps+ t\right)< 0\right)=0.
\end{equation}
Let
\[
\nu^N\stackrel{\text{def.}}{=} \inf\bigg\{t\geq \alpetu{1}/2:\, Y_1^N(N^t) < \alpetu{1} - t-{\eps}/{2}\bigg\}.
\]
By Relation~\eqref{initial L1}, necessarily $\nu^N{>}\alpetu{1}/2$ with probability tending to $1$ as $N$ becomes large.  On the event $\{\nu^N<\gamma\}$ and on the time interval  $[N^{\nu^N},N^{\nu^N+\eps/4}]$, the process $(L^N_1(s))$ is stochastically bounded from below by $(\lfloor N^{\alpetu{1}-\nu^N-\eps/2}\rfloor {-} X_{2,+}(s{-}\nu^N))$, where $(X_{2,+}(s))$ is a birth and death process starting at $0$ and reflected at $0$, for which the transition $x\mapsto x{-}1$ occurs at rate $\lambda_1$ and $x\mapsto x{+}1$ at rate
\[
\mu_1\, \frac{\alpetu{1}-\nu^N -\eps/2+\nu^N+\eps/4}{1 + \alpetu{1}-\eps/4} = \lambda_1 -C \eps
\]
for some constant $C>0$.  Consequently, setting $\gamma_0 =(\alpetu{1}/2+\eps/4)\wedge \gamma$ and using Proposition~\ref{King}, one obtains
\begin{align*}
\P\left(\rule{0mm}{5mm}\right. &\left. \inf_{\alpetu{1}/2\leq t\leq \gamma_0} Y_1^N(N^t)-\alpetu{1} + t+{\eps}\leq 0\right)\notag\\
 &\qquad  =\P\left(\inf_{\alpetu{1}/2\leq t\leq \gamma_0}\left(Y_1^N(N^t)-\alpetu{1} + t+{\eps}\right)\leq 0, {\nu^N}<(\alpetu{1}/2+\eps/4)\wedge \gamma\right)\notag \\
 & \qquad \leq \P\bigg(\sup_{\nu^N\leq t\leq  \gamma_0} \bigg(X_{2,+}(N^t{-}N^{\nu^N})-\lfloor N^{\alpetu{1}-\tau^N-\eps/2}\rfloor+\lceil N^{\alpetu{1}-t-\eps}\rceil\bigg) \geq  0\bigg)\notag \\
&\qquad \leq \P\bigg(\sup_{s\leq N^{\gamma_0}} X_{2,+}(s) {\geq}   N^{\alpetu{1}-\gamma-\eps/2}\left(1{-}N^{-\eps/2}\right) \bigg)\\
  &\qquad \leq \big(\lambda_1 N^{\gamma_0}+1\big)\left(\frac{\lambda_1{-}C\eps}{\lambda_1}\right)^{N^{\alpetu{1}-\gamma-\eps/2}/2}\hspace{-18mm}.
\end{align*}
The quantity in the right hand side of the last relation provides an upper bound on the probability that an excursion of $(Y_1(N^t))$ exceeds $\alpetu{1} - t-{\eps}$ for $t$ in the time  interval $(\alpetu{1}/2, \alpetu{1}/2+\eps/4)$. By repeating the procedure a finite number of times to cover the time interval $(N^{\alpetu{1}/2},N^{\gamma})$, one finally obtains Relation~\eqref{eqauxS}.

Similar arguments show that
\[
\lim_{N\rightarrow \infty} \P\bigg(\sup_{s\in [N^{\alpetu{1}/2},N^{\gamma}]} Y_1^N(s)+Y_2^N(s)> \alpetu{1} +\eps\bigg) = 0,
\]
and the lemma is proved.
\end{proof}
Lemmas~\ref{lem: first bounds} and ~\ref{lem: uniform estimate} show that,  for $t\in(\alpetu{1}/2, \alpetu{2}/2 \wedge\alpetu{1}\wedge 1)$, $L^N_2(N^t)$ is of the order of $N^t$ while $L^N_1(N^t) {\sim}N^{\alpetu{1}-t}$. The following proposition gives a  more precise result, on the (larger) interval $(\alpetu{1}/2,\alpetu{1}\wedge 1)$.
\begin{proposition}\label{prop: phase 2}
For the convergence of processes, the relation
\[
\lim_{N\rightarrow \infty}\bigg(Y_1^N(N^t),\, \frac{L^N_2(N^t)}{N^t}\bigg) = \big(\alpetu{1}-t,\, \mu_2(\rho_2-\rho_1)\big)
\]
holds on the time interval $(\alpetu{1}/2,\alpetu{1}\wedge 1)$.
\end{proposition}
\begin{proof}
One shows instead the equivalent statement
\[
\lim_{N\rightarrow \infty} \bigg(Y_1^N(N^t)+Y_2^N(N^t),\, \frac{L^N_2(N^t)}{N^t}\bigg)= \big(\alpetu{1},\, \mu_2(\rho_2-\rho_1)\big).
\]
Let us start by fixing a constant $\gamma$ such that
\[
\frac{\alpetu{1}}{2}<\gamma < \frac{\alpetu{2}}{2}\wedge \alpetu{1}\wedge 1,
\]
and show the desired convergence on the interval $(\alpetu{1}/2,\gamma]$. The convergence of the first coordinate is then a direct consequence of Lemma~\ref{lem: uniform estimate}.

As a first step, one shows that for any $t\in (\alpetu{1}/2,\gamma]$, the convergence
\begin{equation}
\lim_{N\to+\infty} \frac{L^N_2(N^t)}{N^t} =  \mu_2(\rho_2-\rho_1) \label{L2 conv Q2}
\end{equation}
holds for the $L_2$-norm.

Define the function $F$ by
\begin{equation}\label{approx F}
F(l_1,l_2,t) \stackrel{\text{def.}}{=} \frac{1}{2}\bigg(\frac{l_2}{N^t}-\mu_2(\rho_2-\rho_1)\bigg)^2 -\frac{\mu_2}{\mu_1}\frac{l_1}{N^t}\, \bigg(\frac{l_2}{N^t}-\mu_2(\rho_2-\rho_1)\bigg).
\end{equation}
By using the SDE's~\eqref{SDE1}, trite calculations give that the infinitesimal generator $G^N$ associated to the process  $(L^N_1(N^t),L^N_2(N^t),t)$ is given by
\begin{multline*}
G^NF(l_1,l_2,t) = -(\log N)\bigg(\frac{l_2}{N^t}-\mu_2(\rho_2-\rho_1)\bigg)^2  +C_1^N(l_1,l_2,t)\, \frac{\log N}{N^t} \\+ C_2^N(l_1,l_2,t)\frac{l_1}{N^t}\log N + C_3^N(l_1,l_2,t)\frac{l_1}{N^t}\frac{l_2}{N^t}\log N,
\end{multline*}
and there exists $K>0$ such that, for $i{\in}\{1,2,3\}$, the relation $|C_i^N(l_1,l_2,t)|\leq K$ holds for any $l_1$, $l_2\in \N$, $t\in (\alpetu{1}/2,\gamma]$ and $N\in \N$.  Define
\[
\psi^N(l_1,l_2,t) = C_1^N(l_1,l_2,t) \frac{\log N}{N^t} + C_2^N(l_1,l_2,t)\frac{l_1}{N^t}\log N + C_3^N(l_1,l_2,t)\frac{l_1}{N^t}\frac{l_2}{N^t}\log N.
\]
Then
\begin{multline}\label{MartApp}
  (M^N(t))\stackrel{\text{def.}}{=}\left(\rule{0mm}{6mm}F\left(L^N_1(N^t),L^N_2(N^t),t\right)  - F\left(L^N_1(N^{\alpetu{1}/2}),L^N_2(N^{\alpetu{1}/2}),\alpetu{1}/2\right)\right. \\
\left. {+}\int_{\alpetu{1}/2}^t\left[\log N\bigg(\frac{L^N_2(N^u)}{N^u}{-}\mu_2(\rho_2{-}\rho_1)\bigg)^2{-}\psi^N\left(L^N_1(N^u),L^N_2(N^u),u\right)\right]\diff u\right)
\end{multline}
is a zero-mean martingale on the time interval $(\alpetu{1}/2,\gamma]$. Taking the expectation and reordering the terms conveniently, one obtains that
\begin{multline*}
\!\E\bigg(\bigg(\frac{L^N_2(N^t)}{N^t}{-}\mu_2(\rho_2{-}\rho_1)\bigg)^2\bigg){=}{-} 2(\log N) \int_{\alpetu{1}/2}^t \!\!\E\bigg(\bigg(\frac{L^N_2(N^u)}{N^u}{-}\mu_2(\rho_2{-}\rho_1)\bigg)^2\bigg)\diff u \\
{+} 2\E\Big(\!F\big(L^N_1(N^{\alpetu{1}/2}){,}L^N_2(N^{\alpetu{1}/2}),\alpetu{1}/2\big)\Big) {+} \frac{2\mu_2}{\mu_1}\E\bigg(\!\frac{L^N_1(N^t)}{N^t} \bigg(\!\frac{L^N_2(N^t)}{N^t}{-} \mu_2(\rho_2{-}\rho_1)\bigg)\bigg) \\
+ 2\int_{\alpetu{1}/2}^t \E\big(\psi^N(L^N_1(N^u),L^N_2(N^u),u)\big)\diff u.
\end{multline*}
Now, recall the process $(X_+(s))$ introduced in the proof of Lemma~\ref{lem: first bounds}. A slight adaptation of the arguments given there shows that one can couple $(L^N_1(s))$ and $(2A\lambda_1N^{\alpetu{1}/2} + X_+(s))$ in such a way that $L_1^N(s)\leq 2A\lambda_1N^{\alpetu{1}/2} + X_+(s-N^{\alpetu{1}/2})$ for every $s\in [N^{\alpetu{1}/2},N^\gamma]$. Using this fact together with the Cauchy-Schwartz inequality, one can write
\begin{align*}
 & \left|\E\bigg(\frac{L^N_1(N^t)}{N^t} \bigg(\frac{L^N_2(N^t)}{N^t}- \mu_2(\rho_2-\rho_1)\bigg)\bigg)\right|\\
&  \leq \left[\E\bigg(\bigg(\frac{L^N_1(N^t)}{N^t}\bigg)^2\bigg)\E\bigg(\bigg(\frac{L^N_2(N^t)}{N^t}- \mu_2(\rho_2-\rho_1) \bigg)^2\bigg)\right]^{1/2}\\
&       {\leq} \left[\E\bigg(\bigg(\frac{2A\lambda_1N^{\alpetu{1}/2}{+}X_+(N^t{-}N^{\alpetu{1}/2})}{N^t}\bigg)^2\bigg)\E\bigg(\bigg(\frac{\mathcal{N}_{\lambda_2}[0,N^t]}{N^t}\bigg)^2\!\!{+}\mu_2^2(\rho_2{-}\rho_1)^2\bigg)\right]^{1/2}\\
&       \leq C_4 \left[\E\left(\left[\frac{1}{N^t}\left(2A\lambda_1N^{\alpetu{1}/2}{+}\sup_{\alpetu{1}/2\leq u\leq \gamma}X_+(N^u{-}N^{\alpetu{1}/2})\right)\right]^2\right)\right]^{\mathclap{\quad 1/2}} = C_5\, N^{{\alpetu{1}}/{2}-t},
\end{align*}
by Proposition~\ref{King}, where the constant $C_5$ is independent of $t$ and $N$. Likewise, there exist some constants $C_6$ and $C_7$ such that for any $u\in [\alpetu{1}/2,\gamma)$,
\[
\E\big(|\psi^N(L^N_1(N^u),L^N_2(N^u),u)|\big)\leq C_6\frac{\log N}{N^{\alpetu{1}/2}} +C_7\frac{\log N}{N^{u-{\alpetu{1}}/{2}}}.
\]
Since
\[
(\log N) \int_{\alpetu{1}/2}^t \frac{\diff u}{N^{u-{\alpetu{1}}/{2}}} = 1- \frac{1}{N^{t-{\alpetu{1}}/{2}}},
\]
the process
\[
\left(2\int_{\alpetu{1}/2}^t \E\left(|\psi^N\left(L^N_1(N^u),L^N_2(N^u),u\right)|\right)\diff u\right)
\]
is bounded by a constant uniformly in $t\in [\alpetu{1}/2,\gamma]$ and $N$. Finally, similar arguments give the existence of a constant $C_8$ such  that
\[
\E\Big(|F(L^N_1(N^{\alpetu{1}/2}),L^N_2(N^{\alpetu{1}/2}),{\alpetu{1}/2})|\Big) \leq C_8.
\]
One can now use Gronwall's Lemma to conclude that there exists $C_9>0$ independent of $N$ such that for every $t\in [\alpetu{1}/2,\gamma]$,
\begin{equation}\label{gronwall 2}
\E\bigg(\bigg(\frac{L^N_2(N^t)}{N^t}- \mu_2(\rho_2 - \rho_1)\bigg)^2\bigg) \leq C_9\, e^{-2(\log N) (t-\alpetu{1}/2)} = \frac{C_9}{N^{2t-\alpetu{1}}},
\end{equation}
which proves (\ref{L2 conv Q2}).

As a second step, one now shows the uniform convergence of $(L^N_2(N^t)/N^t)$ towards the constant process $\mu_2(\rho_2-\rho_1)$, over any time interval of the form $[{\alpetu{1}}/{2}+\eps,\gamma]$.
By Doob's maximal inequality applied to the martingale $(M^N(t))$ defined by Relation~\eqref{MartApp}, one has for every $\eta>0$
\begin{align*}
&\P\bigg(\sup_{\alpetu{1}/2+\eps\leq t\leq \gamma}|M^N(t)|>\eta\bigg) \leq \frac{1}{\eta}\, \E\left(|M^N(\gamma)|\right) \\
& \leq \frac{1}{\eta}\, \E\left(\left(\frac{L^N_2(N^\gamma)}{N^\gamma}-\mu_2(\rho_2-\rho_1)\right)^2\right) + \frac{1}{\eta}\, \E\left(\left(\frac{L^N_2(N^{\alpetu{1}/2+\eps})}{N^{\alpetu{1}/2+\eps}}-\mu_2(\rho_2-\rho_1)\right)^2\right) \\
&\hspace{4cm} + \frac{\log N}{\eta}\int_{\alpetu{1}/2+\eps}^\gamma \E\bigg(\bigg(\frac{L^N_2(N^u)}{N^u}-\mu_2(\rho_2-\rho_1)\bigg)^2\bigg)\, \diff u \\
&\hspace{5cm} + \frac{1}{\eta}\int_{\alpetu{1}/2+\eps}^\gamma \E\big(|\psi^N(L^N_1(N^u),L^N_2(N^u),u)|\big)\, \diff u.
\end{align*}
The quantity in the right hand side above converges to $0$ as $N$ tends to infinity, by all the estimates obtained so far and by using Lebesgue's  convergence theorem. Consequently, using Relations~\eqref{approx F} and~\eqref{MartApp}, one obtains that
\begin{align*}
\P\bigg(&\sup_{\alpetu{1}/2+\eps\leq t\leq \gamma}\bigg(\frac{L^N_2(N^t)}{N^t}-\mu_2(\rho_2-\rho_1)\bigg)^2>6\eta\bigg)\\
& \leq \P\bigg(\sup_{\alpetu{1}/2+\eps\leq t\leq \gamma}|M^N(t)|>\eta\bigg)
 + \P\bigg(\bigg(\frac{L^N_2(N^{\alpetu{1}/2+\eps})}{N^{\alpetu{1}/2+\eps}}-\mu_2(\rho_2-\rho_1)\bigg)^2>\eta\bigg) \\
 & \qquad \qquad +\P\bigg(\sup_{\alpetu{1}/2+\eps\leq t\leq \gamma} \frac{\mu_2}{\mu_1}\frac{L^N_1(N^t)}{N^t}\, \bigg|\frac{L^N_2(N^t)}{N^t}-\mu_2(\rho_2-\rho_1)\bigg|>\eta\bigg)\\
& \qquad\qquad+ \P\bigg(\log N\, \int_{\alpetu{1}/2+\eps}^\gamma \bigg(\frac{L^N_2(N^u)}{N^u}-\mu_2(\rho_2-\rho_1)\bigg)^2\, \diff u >\eta\bigg)\\
& \qquad\qquad + \P\bigg(\int_{\alpetu{1}/2+\eps}^\gamma |\psi^N(L^N_1(N^u),L^N_2(N^u),u)|\, \diff u>\eta \bigg).
\end{align*}
Again, by the Markov inequality and the estimates obtained before, each of the six terms in the right hand side of the last relation converges to $0$ as $N$ becomes large. This shows the desired uniform convergence on the time interval $(\alpetu{1}/2,\gamma]$.

The third and last step extends the convergence result to the whole interval $(\alpetu{1}/2,\alpetu{1}\wedge 1)$. Let $\eps \in (0,\alpetu{2}-\alpetu{1})$ and $\eta \in (0,\eps)$.
Let $\nu_0^N$ be defined by
\[
\nu_0^N\stackrel{\text{def.}}{=} \inf\{s\geq \gamma:\, L^N_2(N^s)\leq N^{s-\eta}\}.
\]
The results of the first part of this proof show that, for some constant $C_{10}$,
\[
\lim_{N\to+\infty} \P\left( L^N_2(N^\gamma)\geq C_{10}N^{\gamma}, L^N_1(N^\gamma)\in [N^{\alpetu{1}-\gamma-\eps}, N^{\alpetu{1}-\gamma+\eps}]\right)=1.
\]
Thus, $\nu_0^N>\gamma$ and the  process $(L^N_1(s),N^\gamma\leq s \leq N^{\nu_0^N})$ is stochastically bounded from above by $\lceil N^{\alpetu{1}-\gamma+\eps}\rceil + X_{3,+}(\cdot - N^{\gamma})$, where $(X_{3,+}(s))$ is a birth and death process reflected at $0$ for which the transition $x\mapsto x+1$ occurs at rate $\lambda_1$ and $x\mapsto x-1$ at rate
\[
\mu_1\, \frac{(\alpetu{1}-\gamma+\eps) + (\gamma-\eta)}{1+\alpetu{1}-\gamma+\eps + \gamma-\eta} = \lambda_1+\mu_1\, \frac{\eps -\eta}{(1+\alpetu{1})(1+\alpetu{1}+\eps-\eta)},
\]
since $\mu_1\alpetu{1}/(1+\alpetu{1})=\lambda_1$.  Since $\eps<\eta$, the  drift of $X$ is negative and one can conclude that
\[
\lim_{N\rightarrow\infty}\P\bigg(\sup_{t\in [\gamma,\nu_0^N]}L^N_1(N^t)\geq 2N^{\alpetu{1}-\gamma+\eps}\bigg) =0.
\]
As a consequence, the process  $(L^N_2(N^\gamma+s),0\leq s\leq N^{\nu_0^N}-N^\gamma)$ is stochastically bounded from below by the birth and death process $(X_4(s))$,  with $X_4(0)=C_{10}N^\gamma$,  birth rate $\lambda_2$ and transitions $x\mapsto x-1$ occurring at rate
\[
\mu_2\frac{(\alpetu{1}-\gamma+\eps)+t}{1+\alpetu{1}-\gamma+\eps+t}.
\]
Let $\Delta\stackrel{\text{def.}}{=} \alpetu{2}-(\alpetu{1}+\eps)$. By our choice of $\eps$, $\Delta>0$ and
\[
\lambda_2> \mu_2\frac{(\alpetu{1}-\gamma+\eps)+t}{1+\alpetu{1}-\gamma+\eps+t} \quad \hbox{as long as}\quad t-\gamma \leq \Delta.
\]
 Hence, there exists $C_{11}>0$ such that
\[
\lim_{N\rightarrow \infty}\P\bigg(\inf_{t\in [\gamma,\gamma+\Delta]}\frac{L^N_2(N^t)}{N^t}<C_{11}\bigg) = 0,
\]
and in particular $\nu_0^N> \gamma + \Delta$. Now that it has been proved that, with probability tending to $1$, $L^N_2(N^t)\geq C_{11} N^t$ for any $t\in[\gamma,\gamma+\Delta]$, one can adapt Lemmas~\ref{lem: first bounds} and~\ref{lem: uniform estimate} and the first part of the proof of Theorem~\ref{th: averaging} to show that the uniform convergence holds on the time interval $(\alpetu{1}/2,(\gamma+\Delta)\wedge \alpetu{1}\wedge 1)$ too.

Finally, since the definition of $\Delta$ does not depend on $\gamma$, one can proceed by induction (in finitely many steps) and conclude that
\[
\lim_{N\rightarrow \infty}\bigg(\bigg(Y_1^N(N^t)+Y_2^N(N^t),\frac{L^N_2(N^t)}{N^t}\bigg)\bigg)=(\alpetu{1},\mu_2(\rho_2-\rho_1))
\]
 for the convergence in distribution of processes on $(\alpetu{1}/2,\alpetu{1}\wedge 1)$. The proposition is proved.
\end{proof}
Proposition~\ref{prop: phase 2} establishes the behaviour of $L^N_2$ stated in Theorem~\ref{th: averaging}. There remains to show that, on this interval of time, the convergence
\[
\lim_{N\to+\infty} \left(\frac{L_1^N(N^t)}{N^{\alpetu{1}-t}}\right)=\left(\frac{1}{\mu_2(\rho_2-\rho_1)}\right)
\]
holds. Equivalently, one shows the following important result.
\begin{proposition}\label{prop: averaging}
 The convergence in distribution
\[
\lim_{N\rightarrow \infty} \bigg(\frac{L^N_1(N^t)L^N_2(N^t)}{N^{\alpetu{1}}}\bigg) = (1)
\]
holds on the time interval  $(\alpetu{1}/2,\alpetu{1} \wedge 1)$.
\end{proposition}
The proof uses several technical lemmas, whose proofs are postponed until the end of the proof of Proposition~\ref{prop: averaging}.
It is based on the idea that, for a given value of $L^N_2$, if $L^N_1$ moves too far apart from its equilibrium value $N^{\alpetu{1}}/L^N_2$ (corresponding to the point where the drift of $L^N_1$ cancels), then it is driven back to this value in much less time than $L^N_2$ needs to change.

More precisely, for $\eta\in (0,\alpetu{1}/2)$, the time interval $[N^{\alpetu{1}/2+\eta},N^{(\alpetu{1}\wedge 1)-\eta}]$ can be covered by $N^{\alpetu{1}/2}$ (at most) sub-intervals of length $N^{\alpetu{1}/2}$.  One will first  consider an interval of the form $[T_N,T_N+N^{\alpetu{1}/2}]$ and $\eps{>}0$ such that $L^N_1L^N_2(T_N){<}(1{+} 3\eps/{2})N^{\alpetu{1}}$.  As one will see, with probability tending to $1$  the process $(L^N_1L^N_2(s))$ does not exceed the value $(1+2\eps)N^{\alpetu{1}}$ on the time interval $[T_N,T_N+N^{\alpetu{1}/2}]$.

In a second step, one will show that $L^N_1L^N_2(T_N+N^{\alpetu{1}/2}){\leq}(1{+}3\eps/{2})N^{\alpetu{1}}$, so that  the same result can be applied to the next time interval of width $N^{\alpetu{1}/2}$.

First, let $\tau^N(T_N)$ be the stopping time defined by
$$
\tau^N(T_N) \stackrel{\text{def.}}{=} \inf \left\{s\geq T_N\, :\, L^N_1(s)L^N_2(s) \leq (1+\eps)N^{\alpetu{1}} \right\}\wedge (T_N+N^{\alpetu{1}/2}).
$$
Of course, when $L^N_1L^N_2(T_N)\leq (1+\eps)N^{\alpetu{1}}$ one has $\tau^N(T_N)=T_N$. On the other hand, when $L^N_1L^N_2(T_N)/N^{\alpetu{1}} \in ( 1+\eps, 1+3\eps /{2})$, the following lemma controls the probability that $L^N_1L^N_2$ exceeds $(1+2\eps)N^{\alpetu{1}}$ on the time interval $[T_N,\tau^N(T_N)]$.
\begin{lemma}\label{lem: first excursion}
Assume that $L^N_1L^N_2(T_N)/N^{\alpetu{1}} \in ( 1+\eps, 1+3\eps/{2})$. Then, there exists a constant $C_1>0$ which is independent of the value of $L^N_1L^N_2(T_N)$ and such that
\[
\P\left(\sup_{s\in [T_N,\tau^N(T_N)]} L^N_1L^N_2(s)\geq (1+2\eps)N^{\alpetu{1}}\right) \leq \exp\bigg(- C_1\eps \frac{N^{\alpetu{1}}}{T_N\log N}\bigg).
\]
\end{lemma}
The next lemma shows that $\tau^N(T_N)$  is negligible compared to $N^{\alpetu{1}/2}$.
\begin{lemma}\label{lem: length first excursion}
Suppose that $L^N_1L^N_2(T_N)/N^{\alpetu{1}} \in ( 1+\eps, 1+3\,\eps/2)$. Then there exists $C_2>0$ independent of the  value of $L^N_1L^N_2(T_N)$ and such that
\[
\P\left(\tau^N(T_N) -T_N> \frac{N^{\alpetu{1}/2}}{\log N}\right) \leq  \exp\left(-C_2\eps \frac{N^{\alpetu{1}/2}}{(\log N)^2}\right).
\]
\end{lemma}
The third lemma controls the probability that $(L^N_1L^N_2(s)/N^{\alpetu{1}})$ reaches again the value $(1+3\eps/{2})$ when it starts below $1+\eps$.
\begin{lemma}\label{lem: rest of excursion}
Suppose that $L^N_1L^N_2(T_N)\leq (1+\eps)N^{\alpetu{1}}$. Then, there exists $C_3>0$ and $D_3>0$ independent of the initial value of $L^N_1L^N_2$ and such that
\begin{multline*}
\P\left(\sup_{s\in [T_N,T_N+N^{\alpetu{1}/2}]} L^N_1L^N_2(s) \geq \bigg(1+ \frac{3}{2}\, \eps\bigg)N^{\alpetu{1}}\right) \\ \leq D_3N^{\alpetu{1}/2}\exp\bigg(- C_3\eps \frac{N^{\alpetu{1}}}{T_N\log N}\bigg).
\end{multline*}
\end{lemma}
\noindent
These lemmas are used as follows. If 
\[
{\cal E}_N\stackrel{\text{\rm def.}}{=}\left\{L^N_1L^N_2(T_N) \leq (1+3\eps/{2})N^{\alpetu{1}}\right\},
\]
and if one defines $I_N=[T_N,T_N+N^{\alpetu{1}/2}]$ and $J_N=[\tau^N(T_N),\tau^N(T_N){+}N^{\alpetu{1}/2}]$, then
\begin{align*}
  \P &\bigg(\left\{ \sup_{s\in I_N}  L^N_1L^N_2(s){>} (1{+}2\eps)N^{\alpetu{1}}\right\}\bigcup \left\{  L^N_1L^N_2(T_N{+}N^{\alpetu{1}/2}){\geq} \Big(1{+}\frac{3}{2}\,\eps\Big)\right\}\left| {\cal E}_N\bigg) \right. \\
& \leq\P\bigg(\sup_{s\in [T_N,\tau^N(T_N)]}  L^N_1L^N_2(s){>}(1{+}2\eps)N^{\alpetu{1}}\left| {\cal E}_N\bigg)\right.{+}\P\left.\left(\tau^N(T_N){-}T_N{>} \frac{N^{\alpetu{1}/2}}{\log N}\right| {\cal E}_N \right)\\ & \hspace{5cm} {+}\P\bigg(\sup_{s\in J_N}  L^N_1L^N_2(s){>}\left(1{+}\frac{3}{2}\eps\right)N^{\alpetu{1}}\left| {\cal E}_N\bigg)\right..
\end{align*}
The first term in the right hand side of the above relation is controlled by the inequality of Lemma~\ref{lem: first excursion}, the second term by Lemma~\ref{lem: length first excursion} and the third one by Lemma~\ref{lem: rest of excursion} with $T_N$ replaced by $\tau^N(T_N)$. One finally obtains the existence of a constant $C_4>0$ such that
\begin{multline*}
  \P\bigg(\left\{\sup_{s\in I_N}  L^N_1L^N_2(s)> (1+2\eps)N^{\alpetu{1}}\right\}\\ \bigcup \left\{  L^N_1L^N_2(T_N+N^{\alpetu{1}/2})\geq \Big(1+{3}/{2}\,\eps\Big)\right\}\left| {\cal E}_N \bigg)\right.
  \leq \exp\left(-C_4 \eps \frac{N^{\eta}}{\log N}\right)
\end{multline*}
holds for $N$ sufficiently large. Since $[N^{\alpetu{1} +\eta},N^{(\alpetu{1}\wedge 1)-\eta}]$ can be covered by at most $N^{\alpetu{1}/2}$ intervals of length $N^{\alpetu{1}/2}$, it follows that
\begin{multline*}
\P\bigg(\sup_{s\in [N^{\alpetu{1} +\eta},N^{(\alpetu{1}\wedge 1)-\eta}]} L^N_1L^N_2(s)> (1+2\eps)N^{\alpetu{1}} \bigg)\\  \leq \P\bigg(L^N_1L^N_2(N^{\alpetu{1}/2+\eta})\geq \Big(1+\frac{3}{2}\,\eps\Big)N^{\alpetu{1}}\bigg) + N^{\alpetu{1}/2}e^{-C_4 \eps {N^{\eta}}/{\log N}}.
\end{multline*}
There remains to show that the first term in the right hand side of the inequality just above converges to $0$ as $N$ tends to infinity. This corresponds to the following result.
\begin{lemma}\label{lem: initial condition}
The convergence
\[
\lim_{N\rightarrow \infty}\P\bigg(L^N_1L^N_2(N^{\alpetu{1}/2+\eta})\geq \Big(1+\frac{3}{2}\,\eps\Big)N^{\alpetu{1}}\bigg) = 0
\]
holds.
\end{lemma}
One can then conclude that
\begin{equation}\label{upper bound}
\lim_{N\rightarrow \infty}\P\left(\sup_{s\in [N^{\alpetu{1} +\eta},N^{(\alpetu{1}\wedge 1)-\eta}]} L^N_1L^N_2(s)> (1+2\eps)N^{\alpetu{1}} \right) = 0.
\end{equation}
Similar arguments give an estimation for the lower bound:
\[
\lim_{N\rightarrow \infty}\P\bigg(\inf_{s\in [N^{\alpetu{1} +\eta},N^{(\alpetu{1}\wedge 1)-\eta}]} L^N_1L^N_2(s)< (1-2\eps)N^{\alpetu{1}} \bigg) = 0,
\]
and since this conclusion holds for any $\eps>0$, Proposition~\ref{prop: averaging} is proved.
Combining Propositions~\ref{prop: phase 2} and \ref{prop: averaging}, Theorem~\ref{th: averaging} is proved.

\medskip
\begin{proof}[Proofs of Lemmas~\ref{lem: first excursion} and \ref{lem: rest of excursion}]
Both proofs use the same idea. Let us start by the proof of Lemma~\ref{lem: first excursion}.
Let  $\delta>0$ and define
\[
A_\delta \stackrel{\text{def.}}{=} \left\{\sup_{s\in [N^{\alpetu{1}/2+\eta},N^{(\alpetu{1}\wedge 1)-\eta}]} \bigg|\frac{L^N_2(s)}{s}-\kappa\bigg| \leq \delta \right\},
\]
where $\kappa\stackrel{\text{def.}}{=} \mu_2(\rho_2 - \rho_1)$. By Proposition~\ref{prop: phase 2}, this event has a probability tending to $1$ as $N$ goes to infinity.

Let us work conditionally on the event that $L^N_1L^N_2(T_N)=\lfloor(1+3\eps/{2})N^{\alpetu{1}} \rfloor$.  A simple coupling argument shows that it is enough to consider this case. To ease the notation, one does not report this conditioning in the notation. On the event $A_\delta$, one thus has
\[
L^N_1(T_N) \leq  \ell_N\stackrel{\text{def.}}{=}\frac{(1+3\,\eps/{2})}{\kappa -\delta}\frac{N^{\alpetu{1}}}{T_N}.
\]
Again by a coupling argument, one can assume that $L^N_1(T_N)$ is equal to this upper bound. Note that the relation $L^N_1L^N_2(s){<}(1{+}2\eps)N^{\alpetu{1}}$ holds for any $s\leq \tau^N(T_N)$ if
\begin{equation}\label{target}
\sup_{s\in [T_N,\tau^N(T_N)]}L^N_1(s) <  \frac{(1+2\eps)}{(\kappa + \delta)}\, \frac{N^{\alpetu{1}}}{T_N+N^{\alpetu{1}/2}},
\end{equation}
where the quantity in the denominator is an upper bound on the values taken by $L^N_2$ on $[T_N,T_N+N^{\alpetu{1}/2}]$. Hence, this is what is proved below. Observe that Relation~\eqref{target} is possible for $N$ large enough whenever $\delta$ is chosen small enough so that
\[
\frac{1+{3\eps}/{2}}{\kappa -\delta} < \frac{1+2\eps}{\kappa + \delta}
\]
holds. Now, for $s\in[T_N,\tau^N(T_N)]$ one has
\[
L^N_1(s) \geq \bar{\ell}_N\stackrel{\text{def.}}{=} \frac{(1+\eps)N^{\alpetu{1}}}{(\kappa+\delta)(T_N+N^{\alpetu{1}/2})} \ \ \hbox{and}\ \ L^N_2(s)\geq (\kappa-\delta)T_N,
\]
$(L^N_1(s))$ is therefore  stochastically bounded by $(\ell_N+X_+(s-T_N))$, where $(X_+(s))$ is a  birth and death process reflected at $\overline{\ell}_N{-}\ell_N$,
with birth rate  $\lambda_1$ and a death rate given by
\begin{multline*}
\mu_1 \,\frac{\log\big(\bar{\ell}_N\big) + \log ((\kappa - \delta)T_N)}{\log N + {\log\big(\bar{\ell}_N\big) + \log ((\kappa - \delta)T_N)}} \\= \mu_1\,\frac{\alpetu{1}}{1+\alpetu{1}} + \frac{C_0\log \big({(\kappa-\delta)(1+\eps)}/{(\kappa+\delta)}\big)}{\log N} = \mu_1\,\frac{\alpetu{1}}{1+\alpetu{1}} + \frac{C\eps}{\log N},
\end{multline*}
for some positive constants $C_0$ and $C$. As a consequence, the infinitesimal drift of $X$ is equal to $-C\eps /(\log N)$ and  Proposition~\ref{King} gives  the existence of a constant $C_1>0$ such that
\[
\P\left(\sup_{s\in [0,\tau^N(T_N)-T_N]}X(s) \geq \frac{(1+2\eps)N^{\alpetu{1}}}{(\kappa + \delta)(T_N+N^{\alpetu{1}/2})} -\ell_N\right)  \leq \exp\bigg(-C_1 \eps \frac{N^{\alpetu{1}}}{T_N\log N}\bigg),
\]
which implies~\eqref{target} and proves Lemma~\ref{lem: first excursion}.

The proof of Lemma~\ref{lem: rest of excursion} is similar. Indeed, to obtain the desired upper bound, this time one starts from $L^N_1(T_N){=} {(1{+}\eps)N^{\alpetu{1}}}/{[(\kappa{-}\delta)T_N]}$ and shows that on the time interval $[T_N,T_N+N^{\alpetu{1}/2}]$,  the process  $(L^N_1(t))$ never exceeds the quantity ${(1{+}3\,\eps/{2})N^{\alpetu{1}}}/{[(\kappa{+}\delta)T_N]}$ with a probability that has the required form. The only difference here is that one has to control the number of excursions of $(L^N_1(t))$ above ${(1{+}\eps)N^{\alpetu{1}}}/{[(\kappa{-}\delta)T_N]}$ on the time interval $[T_N,T_N+N^{\alpetu{1}/2}]$. This number is obviously bounded by the number of jumps of size $+1$ performed by $(L^N_1(t))$ during this lapse of time, which itself is stochastically bounded by a Poisson random variable with parameter $\lambda_1 N^{\alpetu{1}/2}$. Thus, for any $C_2>\lambda_1$, there exists $C_3>0$ such that
\[
\P\left(\, \hbox{at least }C_2N^{\alpetu{1}/2}\hbox{ excursions on } [T_N,T_N+N^{\alpetu{1}/2}]\right)\leq e^{-C_3N^{\alpetu{1}/2}}.
\]
Consequently,
\begin{multline*}
\P\bigg(\sup_{s\in [T_N,T_N+N^{\alpetu{1}/2}]}L^N_1L^N_2(s)\geq \Big(1+{3\eps}/{2}\Big)N^{\alpetu{1}}\bigg)  \leq e^{{-}C_3N^{\alpetu{1}/2}}\\{+} C_2N^{\alpetu{1}/2}\exp\bigg({-}C_4\eps\, \frac{N^{\alpetu{1}}}{T_N\log N}\bigg)
 \leq C_5 N^{\alpetu{1}/2}\exp\bigg({-}C_4\eps\, \frac{N^{\alpetu{1}}}{T_N\log N}\bigg)
\end{multline*}
for some positive constants $C_4$ and $C_5$, where the last inequality uses the fact that $N^{\alpetu{1}}/T_N \leq N^{\alpetu{1}/2 -\eta}$. The proof of Lemma~\ref{lem: rest of excursion} is thus complete.
\end{proof}

\begin{proof}[Proof of Lemma~\ref{lem: length first excursion}]
The worst case to consider here is when
\[
L^N_1(T_N)=\tilde{\ell}_N\stackrel{\text{def.}}{=}\frac{(1+{3\,\eps}/{2})N^{\alpetu{1}}}{(\kappa-\delta)T_N}.
\]
Since $L^N_2(s)\leq (\kappa+\delta)(T_N+N^{\alpetu{1}/2})$ on the time interval considered, the probability to estimate is bounded from above by the probability that $(L^N_1(t))$ does not go below 
\[
m_N\stackrel{\text{def.}}{=}\frac{(1+\eps)N^{\alpetu{1}}}{[(\kappa+\delta)(T_N+N^{\alpetu{1}/2})]}
\]
on the time interval $[T_N,T_N+N^{\alpetu{1}/2}/(\log N)]$.
Using the same type of coupling as before, on $[T_N,\tau^N(T_N)]$, $(L^N_1(t))$ is stochastically bounded by $(\tilde{\ell}_N + X_+(t-T_N))$, where $(X_+(t))$ is a birth and death process  reflected at $m_N{-}\tilde{\ell}_N<0$ with birth rate $\lambda_1$ and death rate
\[
\mu_1\ \frac{\log m_N + \log((\kappa-\delta)T_N)}{\log N + \log m_N + \log((\kappa-\delta)T_N)} = \mu_1\, \frac{\alpetu{1}}{1+\alpetu{1}} + \frac{C\eps}{\log N}
= \lambda_1 + \frac{C\eps}{\log N},
\]
where $C>0$. One denotes by $(X(s))$ the non-reflected birth and death process with the same initial point. In particular, $(X(s))$ is a random walk whose drift is equal to $-C\eps /(\log N)$. Consequently,
\begin{align*}
&  \P\left(\tau^N(T_N)-T_N>\rule{0mm}{5mm}   \frac{N^{\alpetu{1}/2}}{\log N}\right) \leq \P\left(\inf_{s\in [0,N^{\alpetu{1}/2}/\log N]} X_+(s)> m_N - \tilde{\ell}_N\right)\\
& \leq \P\left( X\left(\frac{N^{\alpetu{1}/2}}{\log N}\right)> m_N - \tilde{\ell}_N\right) \\& = \P\left(X\left(\frac{N^{\alpetu{1}/2}}{\log N}\right){+} C\eps\frac{N^{\alpetu{1}/2}}{(\log N)^2} {>} m_N {-} \tilde{\ell}_N{+}C\eps\frac{ N^{\alpetu{1}/2}}{(\log N)^2}\right) 
 \leq \exp\left(-C_2\eps \frac{N^{\alpetu{1}/2}}{(\log N)^2}\right)
\end{align*}
for some $C_2>0$, where the last line uses standard large deviations principles applied to the centered random walk
\[
\left(X(t){+} \frac{C\eps }{(\log N)}t\right)
\]
and the fact that
\[
\frac{N^{\alpetu{1}}}{T_N}=o\left(\frac{N^{\alpetu{1}/2}}{(\log N)^2}\right) \quad \text{ \rm implies }\quad  \left|m_N{-}\tilde{\ell}_N\right|=o\left(\frac{N^{\alpetu{1}/2}}{(\log N)^2}\right).
\]
\end{proof}

\begin{proof}[Proof of Lemma~\ref{lem: initial condition}]
The proof is a combination of the arguments used in the proofs of Lemmas~\ref{lem: first excursion}, \ref{lem: length first excursion} and \ref{lem: rest of excursion}.
Indeed, let $s_\eta$ be defined by
$$
N^{s_\eta} = N^{\alpetu{1}/2+\eta}{-}N^{\alpetu{1}/2+\eta/2},\quad \hbox{i.e.}\quad  s_\eta = \frac{\alpetu{1}}{2}+\eta + \frac{\log(1-N^{-\eta/2})}{\log N}.
$$
Since $s_\eta>\alpetu{1}/2$, by Theorem~\ref{th: averaging}, for a given small $\delta>0$, the event 
\begin{multline*}
\left\{L^N_1(N^{s_\eta}) \in \left[N^{\alpetu{1}-s_\eta-\delta},N^{\alpetu{1}-s_\eta+\delta}\right] = \left[\frac{N^{{\alpetu{1}}/{2}-\eta - \delta}}{1-N^{-\eta/2}}, \frac{N^{{\alpetu{1}}/{2}-\eta + \delta}}{1-N^{-\eta/2}}\right]\right\}\\
\bigcup\left\{L^N_2(s)\in \left[(\kappa-\delta)N^{s_\eta},(\kappa+\delta)N^{\alpetu{1}/2+\eta}\right], \,\forall s\in [N^{s_\eta},N^{\alpetu{1}/2+\eta}]\right\}
\end{multline*}
has a probability converging to $1$ as $N$ becomes large. Recall that $\kappa= \mu_2(\rho_2 - \rho_1)$. As before, via a coupling, one can assume  that $L^N_1(N^{s_\eta})$ is equal to the maximal value $N^{\alpetu{1}-s_\eta+\delta}$. For $\eps>0$, define
\[
\ell_N\stackrel{\text{def.}}{=} \frac{1+\eps}{\kappa +\delta}\, N^{\alpetu{1}/2-\eta} \text{ and }\sigma_N\stackrel{\text{def.}}{=} \inf\left\{s\geq N^{s_\eta}:\, L^N_1(s)\leq \ell_N\right\}.
\]
 One first shows that $\sigma_N{<}N^{\alpetu{1}/2+\eta}$ holds with probability tending to $1$ as $N$ becomes large. On the time interval $[N^{s_\eta}, \sigma_N]$, the process $(L^N_1(s))$ is stochastically bounded by
\[
\left(\frac{N^{{\alpetu{1}}/{2}-\eta + \delta}}{1-N^{-\eta/2}}+ X(s{-}N^{s_\eta})\right),
\]
where $(X(s))$ is a birth and death process on $\Z$ starting at $0$ with birth rate $\lambda_1$ and a death rate given by
\begin{multline*}
  \mu_1\,\frac{\log\big(\ell_N\big) + \log ((\kappa-\delta)N^{\alpetu{1}/2+\eta}(1-N^{-\eta/2}))}{\log N +\log\big(\ell_N\big) + \log ((\kappa-\delta)N^{\alpetu{1}/2+\eta}(1-N^{-\eta/2}))}\\
  = \mu_1\, \frac{\alpetu{1}}{1+\alpetu{1}}+ \frac{C\eps}{\log N}=\lambda_1+ \frac{C\eps}{\log N},
\end{multline*}
for some constant $C>0$. Hence, as in the proof of Lemma~\ref{lem: length first excursion}, one has
\begin{multline*}
\P\left(\sigma_N> N^{\alpetu{1}/2+\eta}\right)  \leq \P\left(\inf_{s\in [0,N^{\alpetu{1}/2+\eta/2}]}X(s)> \frac{1+\eps}{\kappa +\delta}\, N^{\alpetu{1}/2-\eta} - \frac{N^{\frac{\alpetu{1}}{2}-\eta + \delta}}{1-N^{-\eta/2}} \right) \\
 \leq \P\left( X(N^{\alpetu{1}/2+\eta/2}) + \frac{C\eps N^{\alpetu{1}/2+\eta/2}}{\log N} > -C_1 N^{\alpetu{1}/2-\eta+\delta} + \frac{C\eps N^{\alpetu{1}/2+\eta/2}}{\log N} \right).
\end{multline*}
This last term converges to $0$ as $N$ tends to infinity whenever $\delta<3\eta/2$, since then $N^{\alpetu{1}/2-\eta+\delta}$ is negligible compared to $N^{\alpetu{1}/2+\eta/2}/\log N$. As before, one uses standard large deviation estimates on centered random walks.

Secondly, one can see that conditionally on the event $\{\sigma_N{<}N^{\alpetu{1}/2+\eta}\}$, the process $(L^N_1(s))$ stays below the value ${(1{+}{3\eps}/{2})}N^{\alpetu{1}/2 -\eta}/{(\kappa{+}\delta})$ on the time interval $[\sigma_N,N^{\alpetu{1}/2+\eta}]$ with a probability tending to $1$. It is  proved using exactly the same method as in the proof of Lemma~\ref{lem: rest of excursion}.

The quantity $\delta>0$ is chosen sufficiently small that  $(1{+}\eps)/(\kappa{+}\delta){>}1/\kappa$. One just has to prove that
\[
\lim_{N\to+\infty} \P\left(\sigma_N<N^{\alpetu{1}/2+\eta}, \sup_{T_N\leq s\leq N^{\alpetu{1}/2+\eta}} L^N_1(s)\leq \frac{1+{3\,\eps}/{2}}{\kappa+\delta}N^{\alpetu{1}/2 -\eta}\right)=1,
\]
hence, with probability tending to $1$,
\[
L^N_1L^N_2(N^{\alpetu{1}/2+\eta}) \leq \frac{1+{3\,\eps}/{2}}{\kappa+\delta}\, N^{\alpetu{1}/2 -\eta}\times (\kappa+\delta)N^{\alpetu{1}/2+\eta} = \Big(1+\frac{3}{2}\, \eps\Big)N^{\alpetu{1}}.
\]
Lemma~\ref{lem: initial condition} is proved.
\end{proof}

\subsection{Proof of Proposition~\ref{prop: phase 3}}\label{ss:proof 2}
Again, one starts by establishing some crude bounds on the number of pending requests in nodes $1$ and $2$ over the time interval of interest. Recall the notation $\kappa = \mu_2(\rho_2-\rho_1)$.
\begin{lemma}\label{lem: L1 vanishes}
For $\eps>0$ sufficiently small there exists $C_\eps>0$ such that
\[
\lim_{N\rightarrow \infty}\P\left(\sup_{s\in I_N}L_1(s)<N^\eps, \inf_{s\in I_N} \frac{L_2(s)}{s}\geq C_\eps\right) = 1,
\]
with $I_N=[N^{\alpetu{1}-\eps/2},N^{(\alpetu{2}\wedge 1)-2\eps}]$. 
\end{lemma}

\begin{proof}[Proof of Lemma~\ref{lem: L1 vanishes}]
Let us define
$$
\sigma_N := \inf\big\{s\geq N^{\alpetu{1} -\eps/2}:\, L^N_1(s)\geq N^\eps\big\}.
$$
One knows from Theorem~\ref{th: averaging} that $\sigma_N>N^{\alpetu{1}-\eta}$ for any $\eta>0$. Besides, a simple coupling argument shows that with probability tending to $1$, $L^N_2(s)\leq 2\lambda_2 s$ for every $s\in[0,N]$. Hence, on the time interval $[N^{\alpetu{1}-\eps/2},\sigma_N]$, $(L_2^N(t))$ is stochastically bounded from below by the birth and death process $(\tilde{L}_2^N(t))$ such that
\[
\tilde{L}_2^N\left(N^{\alpetu{1}-\eps/2}\right) = L^N_2\left(N^{\alpetu{1}-\eps/2}\right)\sim \kappa N^{\alpetu{1}-\eps/2},
\]
and for which transitions $x\mapsto x+1$ occur at rate $\lambda_2$ and $x\mapsto x-1$ at rate
\[
\mu_2\, \frac{\eps {+} [\log(s){+}\log(2\lambda_2)]/\log N}{1{+} \eps {+} [\log(s){+}\log(2\lambda_2)]/\log N} {=} \mu_2\, \frac{\eps {+} {\log s}/{\log N}}{1{+} \eps {+} {\log s}/{\log N}} {+} \frac{C}{\log N}.
\]
But since
\[
\lambda_2> \mu_2\frac{\eps {+} {\log s}/{\log N}}{1{+} \eps {+} {\log s}/{\log N}} 
\]
is equivalent to
\[
\frac{\log s}{\log N} < \frac{\rho_2}{1{-}\rho_2}{-}\eps = \alpetu{2}{-}\eps,
\]
the infinitesimal drift of $(\tilde{L}_2^N(t))$ is bounded from below by some $c_\eps>0$ on the interval $[N^{\alpetu{1}-\eps/2},\sigma_N\wedge N^{(\alpetu{2}\wedge 1)-2\eps}]$. The ergodic theorem for Poisson processes thus guarantees that $\tilde{L}_2^N(s)/s$ remains greater than $C_\eps{=}\kappa{+}c_\eps/2$ with probability tending to $1$ as $N\rightarrow \infty$, and so
\begin{equation}\label{partial result}
\lim_{N\rightarrow \infty}\P\left(\inf_{s\in [N^{\alpetu{1}-\eps/2},\sigma_N\wedge N^{(\alpetu{2}\wedge 1)-2\eps}]} \frac{L_2^N(s)}{s}\geq C_\eps\right) = 1.
\end{equation}

Now, using this first result together with Theorem~\ref{th: averaging}, for $N$ large enough  one can write that on the smaller time  interval $[N^{\alpetu{1}-\eps/4},\sigma_N\wedge N^{(\alpetu{2}\wedge 1)-2\eps}]$, the process $(L_1^N(t))$ is stochastically bounded from above by $N^{\eps/2} + X_+(\cdot \, - N^{\alpetu{1}-\eps/4})$, where $(X_+(t))$ is a birth and death process reflected at $0$, with birth rate $\lambda_1$ and a death rate equal to
\begin{multline*}
\mu_1\, \frac{{\eps}/{2}+\alpetu{1} -{\eps}/{4}+{\log C_\eps}/{\log N}}{1+ {\eps}/{2}+\alpetu{1} -{\eps}/{4}+{\log C_\eps}/{\log N}} \\= \mu_1\, \frac{\alpetu{1}}{1+\alpetu{1}} + C'\eps + \frac{C''}{\log N} = \lambda_1 + C'\eps + \frac{C''}{\log N},
\end{multline*}
for some constants $C'$ and $C''>0$. Hence, b) of Proposition~\ref{King} enables us to conclude that $\sigma_N>N^{(\alpetu{2}\wedge 1) -2\eps}$ holds with probability tending to $1$. Recalling Relation (\ref{partial result}) and the fact that  $\sigma_N>N^{(\alpetu{2}\wedge 1) -2\eps}$ is equivalent to
$$
\sup_{s\in [N^{\alpetu{1}-\eps/2},N^{(\alpetu{2}\wedge 1)-2\eps}]}L_1^N(s)<N^\eps,
$$
Lemma~\ref{lem: L1 vanishes} is proved. \end{proof}

One can now complete the proof of Proposition~\ref{prop: phase 3}.

\begin{proof}[Proof of a) of Proposition~\ref{prop: phase 3}]
  Let $\e>0$. From Lemma~\ref{lem: L1 vanishes}, one knows that with probability tending to 1, $L_1^N(s)$ remains below $N^\e$ and $L_2^N(s)/s$ remains above $C_\e$ on the time interval $[N^{\alpetu{1}-\e/4},N^{(\alpetu{2}\wedge 1)-2\e}]$.

  Hence, on the sub-interval $[N^{\alpetu{1}}\log N,N^{(\alpetu{2}\wedge 1)-2\e}]$, $(L_1^N(t))$ is stochastically bounded from above by $N^\e + X_+(\cdot\, - N^{\alpetu{1}}\log N)$, where $X_+$ is a birth and death process starting at 0, reflected at $-N^\e$, with birth rate $\lambda_1$ and death rate
$$
\mu_1\, \frac{\alpetu{1} + {\log(C_\e\log N)}/{\log N}}{1+ \alpetu{1}+ {\log(C_\e\log N)}/{\log N}}  = \lambda_1 + C\, \frac{\log \log N}{\log N}
$$
for some $C>0$. Standard estimates on random walks thus yield
$$
\lim_{N\rightarrow \infty}\P\big((X_+(t))\hbox{ does not hit }(-N^{\e}) \hbox{ before }N^\e\log N\big) = 0,
$$
from which the result follows. \end{proof}

\begin{proof}[Proof of b) of Proposition~\ref{prop: phase 3}]
The same coupling as in the proof of a) of  Proposition~\ref{prop: phase 3} still holds on the interval $[\theta_0^N,N^{(\alpetu{2}\wedge 1)-2\e}]$ (replacing the initial value $N^\e$ by $0$ and reflecting $X_+$ at $0$ instead of $-N^\e$). In particular, by Proposition~\ref{King}b)
$$
\lim_{N\rightarrow \infty}\P\left((X_+(t)) \hbox{ reaches }\frac{(\log N)^2}{\sqrt{\log\log N}} \hbox{ before time }N^{\alpetu{1}+\e}\right)=0.
$$

Next, since $(L_2^N(s))$ increases linearly on the time interval $[N^{\alpetu{1}+\e},N^{(\alpetu{2}\wedge 1)-2\e}]$, another coupling in which $(X_+(t))$ has infinitesimal drift $-C\e$ (due to the fact that $s\geq N^{\alpetu{1}+\e}$) and the initial value is $(\log N)^2/\sqrt{\log\log N}$) shows that
$$
\lim_{N\rightarrow \infty}\P\left(\sup_{s\in[N^{\alpetu{1}+\e},N^{(\alpetu{2}\wedge 1)-2\e}]} L_1^N(s)  > (\log N)^2\right) = 0,
$$
by Proposition~\ref{King}. These two facts combined give the result. \end{proof}

\begin{proof}[Proof of c) of  Proposition~\ref{prop: phase 3}]
Fix $\e>0$ small. Since $\theta_0^N< N^{\alpetu{1}+\e}$ with probability tending to $1$ by a) of Proposition~\ref{prop: phase 3}, by b) of Proposition~\ref{prop: phase 3} one has that $(L_1^N(t))$ is bounded by $(\log N)^2$ on the time interval  $[N^{\alpetu{1}+\e},N^{(\alpetu{2}\wedge 1)-\e}]$. A proof similar to that of Proposition~2 in \cite{RV} then gives the result.
\end{proof}

\subsection{Proof of Theorem~\ref{th: fluid limit}}\label{ss:proof 3}
Fix $\eps>0$ small. Since $(L_0^N(t))$ is stochastically bounded from above by a Poisson process with rate $\lambda_0$, and from below by $N$ minus a Poisson process with rate $\mu_0$, if $\eta\leq \min\{\eps/(2\lambda_0), \eps/(2\mu_0)\}$ one has
\begin{equation}\label{encadrement}
\lim_{N\rightarrow \infty}\P\left(\sup_{s\in [0,\eta N]}\bigg|\frac{L_0^N(s)}{N}-1\bigg|\leq \eps\right) = 1.
\end{equation}

Suppose the conditions of case~\ref{th2a}) Theorem~\ref{th: fluid limit} are satisfied. It is easy to see that Proposition~\ref{prop: phase 1} holds on the interval $(0,1+(\log \eta)/\log N]$. Hence,
\begin{multline*}
\lim_{N\rightarrow \infty} \P\bigg(\bigg|\frac{L_0^N(\eta N)}{N}{-}1\bigg|\leq \eps, \bigg|\frac{L_1^N(\eta N)}{N}{-}\mu_1\bigg(\rho_1{-}\frac{2}{3}\bigg)\eta \bigg|\leq \eps, \\
 \bigg|\frac{L_2^N(\eta N)}{N}{-}\mu_2\bigg(\rho_2{-}\frac{2}{3}\bigg)\eta \bigg|\leq\eps \bigg) {=} 1.
\end{multline*}
From $\eta N$ on, the processes $(L_0^N(t))$, $(L_1^N(t))$ and $(L_2^N(t))$ are all of the order of $N$. Recalling the definition~\eqref{WL} of the quantity $W(L)$, one can thus conclude that the processes of the number of requests $(L_1^N(t))$ and ($L_2^N(t))$  receive a fraction $2/3$ of the capacity of the channels, while the number of requests in the central node $(L_0^N(t))$ receive $1/3$ of the capacity. By coupling $(L_0^N(\eta N+Nt),L_1^N(\eta N + Nt),L_2^N(\eta N + Nt))$ with the solutions to the system (\ref{SDE1}) starting from the extremal values
\[
(1 \pm \eps, \mu_1(\rho_1-2/3)\eta \pm \eps, \mu_2(\rho_2-2/3)\eta\pm \eps),
\]
one obtains that for any $T\in [\eta,t_0-\eta)$,
\begin{multline*}
\lim_{N\to+\infty} \P\bigg(\sup_{t\in [\eta,T]}\bigg|\frac{L_0^N(Nt)}{N}{-}1{-}\mu_0\bigg(\rho_0{-}\frac{1}{3}\bigg)t\bigg|\leq 2\eps,\\ \sup_{t\in [\eta,T]}\bigg|\frac{L_1^N(Nt)}{N}{-}\mu_1\bigg(\rho_1{-}\frac{2}{3}\bigg)t \bigg|\leq 2\eps, \\
\sup_{t\in [\eta,T]}\bigg|\frac{L_2^N(Nt)}{N}{-}\mu_2\bigg(\rho_2{-}\frac{2}{3}\bigg)t \bigg|\leq 2\eps \bigg)=1.
\end{multline*}
This result shows in particular that $L_0^N(Nt)$ becomes negligible compared to $N$ when $t$ approaches $t_0$, hence the bound on the interval of time considered. Since $\eta$ can be chosen as small as one wants, this proves the desired uniform convergence on $(0,t_0)$.

Suppose now that the conditions of case~\ref{th2b}) of Theorem~\ref{th: fluid limit} are satisfied. Using Relation~\eqref{encadrement}, Theorem~\ref{th: averaging} can be extended to the time interval $[\alpetu{1}/2,1+(\log \eta)/\log N]$. Consequently, with probability tending to $1$, $L_0^N(\eta N)$ and $L_2^N(\eta N)$ are both of the order of  $N$ while $L_1^N(\eta N)$ is of the order of $N^{\alpetu{1}-1}$. Then a close look at the proof of Proposition~\ref{prop: averaging} reveals that $L_1^NL_2^N/N^{\alpetu{1}}$ converges to $1$ as long as $L_2^N$ is of the order of $N$. Consequently, one obtains that, on the time interval of interest, node~$0$ receives a fraction $1/(1+\alpetu{1})=1-\rho_1$ of the capacity of the channel, and nodes~1 and 2 receive a fraction $\alpetu{1}/(1+\alpetu{1})=\rho_1$. Using the coupling with the system (\ref{SDE1}) starting at the extremal values mentioned in the previous paragraph, one can then conclude.

Assuming that the conditions of case~\ref{th2c}) of Theorem~\ref{th: fluid limit} are satisfied, Proposition~\ref{prop: phase 3} can be extended to the time interval $(\alpetu{1},1+(\log \eta)/\log N]$, showing that this time $L_0^N(\eta N)$ and $L_2^N(\eta N)$ are of order $N$ while $L_1^N(\eta N)\leq (\log N)^2$ is negligible compared to any power of $N$. Thus, as long as $(L_0^N(t))$ and $(L_2^N(t))$ remain of order $N$, the same proof as that of b) of  Proposition~\ref{prop: phase 3} guarantees that with probability tending to $1$, $(L_1^N(t))$ remains below $(\log N)^2$. In particular, by the definition~\eqref{WL} of $W(L)$, this means that each of the nodes $0$ and $2$ receives a fraction $1/2$ of the capacity of the channels. The conclusion follows from the same arguments as above (see the proofs of Theorem~4 and Proposition~8 in \cite{RV}) for more details).

Finally, the same reasoning together with Theorem~3 in \cite{RV} prove the result of case~\ref{th2d}) of Theorem~\ref{th: fluid limit}.

\section{General case}\label{SecGen}
In this section, one extends the results of Section~\ref{secsec} to the case $J\geq 3$. As before, each queue $i$, independently of the others, receives new jobs at rate $\lambda_i$ and has an exponential service time with parameter $\mu_i$. Recall that when the Markov process is in state  $L=(L_j)$, for every $i\in \{1,\ldots,J\}$, queue $i$ is served at rate
\[
W(L)={\sum_{j=1}^{J}\log (1{+}L_j)}\left/{\sum_{j=0}^J \log (1{+}L_j)}\right.
\]
while  queue $0$ is served at rate
\[
1{-}W(L)={\log (1{+}L_0)}\left/{\sum_{j=0}^J \log (1{+}L_j)}\right..
\]
The initial state is $(L_0^N(0),\ldots,L_J^N(0))=(N,0, \ldots, 0)$ and  $\alpetu{i}$ denotes ${\rho_i}/{(1-\rho_i)}$, where $\rho_i=\lambda_i/\mu_i$. The nodes with indices greater than or equal to $1$ are numbered so that $\rho_1{<}\rho_2{<}\cdots{<}\rho_J$.

Theorem~\ref{th: fluid general2} at the end of this section summarizes the results obtained on the fluid time scale. Because most of its proof consists in using or slightly adapting the arguments presented in Section~\ref{secsec}, below one only details the reasoning for the particularly interesting case when only queues~0 and $J$ are non trivial in the fluid regime. One will first analyze the network on the time interval $[0,N^{\alpetu{1}/(J-1)}]$, and then on $[N^{\alpetu{1}/(J-1)},+\infty)$. Concerning the first interval, the results are analogous to those obtained in Section~\ref{secsec} and their proofs are very similar (if not identical). For this reason, only the non-obvious modifications will be given. Concerning the second time interval, assuming that $\alpetu{1}/(J{-}1){<}1$, one will show that after time $N^{\alpetu{1}/(J-1)}$, the process $(L_1^N(t))$ remains negligible compared to the sizes of the other queues and therefore does not contribute to $W(L^N)$.  As a consequence,  the impact of queue~1 on the other queues can be ignored,  one is left with a system of $J$ queues, and a simple recurrence then concludes the study.

The results concerning the first phase on the time interval $[0,N^{\alpetu{1}/(J-1)}]$ are the following. Their proofs are sketched towards the end of this section.
\begin{proposition}\label{prop: general}
The convergence in distribution
\[
\lim_{N\rightarrow  \infty} \bigg(\frac{L_1^N(N^t)}{N^t},\, \ldots\, , \frac{L_J^N(N^t)}{N^t}\bigg) = \bigg(\lambda_1{-} \mu_1\, \frac{Jt}{1{+} Jt}, \, \ldots\, , \lambda_J {-} \mu_J\, \frac{Jt}{1{+} Jt}\bigg)
\]
holds on the time interval $\big(0,\alpetu{1}/J\wedge 1\big)$.
\end{proposition}
Next, assuming that $\alpetu{1}/J<1$, once again there exists $\gamma\in (\alpetu{1}/J, (\alpetu{2}/J)\wedge 1)$ such that the infinitesimal drift of each of the queues with index between $2$ and $J$ remains positive up to time $N^{\gamma}$. As in Section~\ref{secsec}, this leads to the following theorem.
\begin{theorem}\label{th: general}
The convergence in distribution
\begin{multline*}
\lim_{N\rightarrow \infty} \bigg(\frac{L_1^N(N^t)}{N^{\alpetu{1} - (J-1)t}},\frac{L_2^N(N^t)}{N^t},\, \ldots\, , \frac{L_J^N(N^t)}{N^t}\bigg) \\
=  \bigg(\prod_{j=2}^J\frac{1}{ \mu_j(\rho_j{-}\rho_1)},\mu_2(\rho_2{-}\rho_1),\, \ldots\, , \mu_J(\rho_J{-}\rho_1)\bigg)
\end{multline*}
holds on the time interval $(\alpetu{1}/J, \alpetu{1}/(J{-}1)\wedge 1)$.
\end{theorem}

Finally, concerning the second phase $[N^{\alpetu{1}/(J-1)},+\infty)$ one has the following analogue  of Proposition~\ref{prop: phase 3}.
\begin{proposition}
  Under the condition $\alpetu{1}/(J-1)<1$, if
  \[
  \theta_0^N=\inf\{t>0: L_1^N(t)=0\},
  \]
  then
\begin{enumerate}
\item for $\eps>0$, one has
\[
\lim_{N\rightarrow \infty} \P\left(\theta_0^N\leq N^{\alpetu{1}/(J-1)}\log N + N^\eps \log N\right) = 1.
\]
\item For  $\eps \in (0,\alpetu{2}/(J-1)\wedge 1)$, one has
\begin{equation}\label{ref2}
\lim_{N\rightarrow \infty} \P\left(\sup_{s\in [\theta_0^N,N^{(\alpetu{2}/(J-1)\wedge 1) -2\eps}]} L_1^N(s) \leq (\log N)^2\right) = 1.
\end{equation}
\item The convergence of processes
\begin{multline}\label{ref3}
\lim_{N\rightarrow \infty}\bigg(\frac{L_2^N(N^t)}{N^t},\ldots,\frac{L_J^N(N^t)}{N^t}\bigg)\\ = \bigg(\lambda_2 - \mu_2\, \frac{(J-1)t}{1+(J-1)t}, \ldots, \lambda_J - \mu_J\, \frac{(J-1)t}{1+(J-1)t}\bigg)
\end{multline}
holds on the time interval $(\alpetu{1}/(J-1),\alpetu{2}/(J-1)\wedge 1)$.
\end{enumerate}
\end{proposition}
Consequently, Relations~\eqref{ref2} and~\eqref{ref3} tell us that from time $N^{\alpetu{1}/(J-1)}$ on,  queue~$1$ does not contribute to the quantity $W(L^N)$. One is thus  left with a network with  $J$ nodes indexed by $0$, $2$, \ldots,~$J$ and starting from the state
\begin{align*}
L_0^N \sim N \text{ and } L_j^N \sim \mu_j(\rho_j -\rho_1) N^{\alpetu{1}/(J-1)},  \quad 2\leq j\leq J.
\end{align*}
Under the condition $\alpetu{2}/(J{-}1)< 1$, at time $N^{\alpetu{2}/(J-1)}$ the infinitesimal drift of $L_2^N$ cancels while the  infinitesimal drifts of the processes $(L_3^N(t))$, \ldots, $(L_J^N(t))$ remain positive for some time. Consequently, the processes $(L_3^N(N^t))$, $\ldots$, $(L_J^N(N^t))$  grow proportionally to $N^t$. At the same time, the product $(L_2^N \cdots L_J^N(N^t))$ remains close to $N^{\alpetu{2}}$, and so $(L_2^N(N^t))$ decreases like $N^{\alpetu{2}-(J-2)t}$. As before, once $(L_2^N(t))$ has reached $0$, it remains below $(\log N)^2$ with probability tending to $1$.  From time $N^{\alpetu{2}/(J-2)}$ on, one is left with a system of $J{-}1$ queues, and so on.

As mentioned earlier, the following theorem describes the most interesting case, in which only queues $0$ and $J$ are non trivial on the fluid time scale. Its proof is similar to that of Theorem~\ref{th: fluid limit} and is therefore omitted.
\begin{theorem}\label{th: fluid general}
Under the condition  $\alpetu{J}{<}1$, the convergence in distribution
\[
\lim_{N\rightarrow \infty}\bigg(\frac{L_0^N(Nt)}{N}, \frac{L_1^N(Nt)}{(\log N)^3}, \ldots, \frac{L_{J-1}^N(Nt)}{(\log N)^3}, \frac{L_J^N(Nt)}{N^{\alpetu{J}}} \bigg) = \big(\gamma_0(t),0, \ldots, 0, \gamma_0(t)^{\alpetu{J}}\big)
\]
holds on the time interval $(0,+\infty)$ with $\gamma_0(t){=} (1{+}\mu_0(\rho_0 {+}\rho_J {-} 1)t)^+$.
\end{theorem}

Before formulating the most general result that can be obtained in the case $J\geq 2$, let us give the main modifications to the proof of Theorem~\ref{th: averaging}  required to prove Theorem~\ref{th: general}. 
\begin{proof}[Sketch of the proof of Theorem~\ref{th: general}] The analogue of  the function $F(\cdot)$   of the proof of Proposition~\ref{prop: phase 2} is given by the functions
\begin{equation}\label{martingale function}
F_j(\ell,t) = \frac{1}{2}\bigg(\frac{l_j}{N^t} {-} \mu_j(\rho_j {-} \rho_1)\bigg)^2 {-} \frac{\mu_j}{\mu_1}\, \frac{l_1}{N^t}\, \bigg(\frac{l_j}{N^t} {-} \mu_j(\rho_j{-}\rho_1)\bigg),
\end{equation}
for $j\geq 2$, $\ell=(l_1,\ldots,l_J)\in \N^J$ and $t\geq 0$. The infinitesimal generator $G^N$ of the Markov process $(L_1(N^t),\ldots, L_J(N^t),t)$ applied to this function $F_j$ yields
\begin{multline*}
G^N(F_j)(\ell, t) = -(\log N) \bigg(\frac{l_j}{N^t}{-}\mu_j(\rho_j{-}\rho_1)\bigg)^2 {+} C_{j,1}^N(\ell, t)\,\frac{l_1l_j}{N^{2t}}\, \log N \\  +  C_{j,2}^N(\ell, t)\,\frac{l_1}{N^t}\, \log N +  C_{j,3}^N(\ell, t)\,\frac{\log N}{N^t},
\end{multline*}
where the functions $C_{j,1}^N(\cdot)$, $C_{j,2}^N(\cdot)$ and $C_{j,3}^N(\cdot)$ are bounded by some constant $K>0$, uniformly in their arguments and in $N\geq 1$. Using the corresponding martingale problem for each $j\in \{2,\ldots,J\}$ separately, the same arguments as in the proof of Proposition~\ref{prop: phase 2} carry over and lead to the uniform convergence of each coordinate. From this, it is straightforward to conclude.
\end{proof}
One concludes by gathering some of the results of the paper into the following theorem. It is restricted to the time interval where the central node is still in the fluid scale regime, i.e. of the order of $N$. The quantity $\kappa$ defined in this theorem is  related to the number of nodes which can be removed without changing the behaviour of the other nodes on the fluid time scale.

If the central node~$0$ becomes empty, the formulation of the results after that instant is not difficult. It corresponds to the case where the central node and a subset of the other nodes are at equilibrium, in the sense that their numbers of requests is $o((\log N)^3)$ on a finite time interval on the fluid time scale.  Analogous results can be stated  when the initial state $L^N(0)$ given by Relation~\eqref{initstate}  is changed in the following way:
\[
L^N(0){=}(L^N_j(0))= N\cdot (\ell_0,\ldots,\ell_J)+o(N),
\]
where $(\ell_j)\in\R_+^{J+1}$ and $\ell_0{+}\cdots{+}\ell_J{=}1$.

\newpage
\begin{theorem}[Convergence on the Fluid Time Scale]\label{th: fluid general2}
Suppose that $\rho_1{<}\rho_2{<}\cdots{<}\rho_J$, recall that
  \[
\alpetu{j}{=}\frac{\rho_j}{1{-}\rho_j},  
\]
and let, for $1\leq j\leq J$,
\[
\beta_j^*\stackrel{\text{\rm def.}}{=}\frac{\alpetu{j}}{J{-}j}, \qquad \kappa \stackrel{\text{\rm def.}}{=} \sup\left\{k: \frac{\alpetu{k}}{J{-}k{+}1}<1\right\}
\]
with the convention that $\sup(\emptyset){=}0$.
Condition~(C) is  that either $\kappa{=}0$ or that  $1{\leq}\kappa{<}J$ and $\beta_\kappa^*{<}1$.

For $j{\geq}1$ and $t{\geq}0$, define
\begin{equation*}\label{gammaf}
(\gamma_0(t), \gamma_j(t)){=}
    \begin{cases}
   \displaystyle   \left(1{+}\mu_0\left(\!\rho_0{-}\frac{1}{J{-}\kappa{+}1}\right)t,  \mu_j\left(\!\rho_j{-}\frac{J{-}\kappa}{J{-}\kappa{+}1}\right)t\right) \text{ if  (C) holds }\\
  \displaystyle    \left(1+\mu_0\left(\rho_0 +\rho_{\kappa}-1 \right)t, \mu_j\left(\rho_j{-}\rho_\kappa\right)t \right), \text{ otherwise,}
    \end{cases}
\end{equation*}
and  
\[
t_0\stackrel{\text{\rm def.}}{=}
\begin{cases}
  \displaystyle \frac{J{-}\kappa{+}1}{\mu_0(1{-}\rho_0(J{-}\kappa{+}1))^+} \text{ if  (C) holds, } \\
  \displaystyle  \frac{1}{\mu_0(1{-}\rho_0{-}\rho_\kappa)^+} \text{ otherwise.}
\end{cases}
\]

The following convergences in distribution of processes hold  on the time interval $(0,t_0)$
  \begin{enumerate}
  \item If $\kappa{=}0$, 
    \[
    \lim_{N\rightarrow \infty}\left(\frac{L_{0}^N(Nt)}{N}, \ldots,\frac{L_{J}^N(Nt)}{N}\right) = (\gamma_{0}(t),\gamma_{1}(t),\ldots,\gamma_J(t));
    \]
  \item In the case $1{\leq}\kappa{<}J$, there are two possible behaviours depending on $\beta_\kappa^*$,\medskip
    \begin{enumerate}
      \item If $\beta_\kappa^*{<}1$, 
\begin{multline*}
\lim_{N\rightarrow \infty}\left(\frac{L_{0}^N(Nt)}{N}, \frac{L_1^N(Nt)}{(\log N)^3}, \ldots, \frac{L_\kappa^N(Nt)}{(\log N)^3}, \frac{L_{\kappa+1}^N(Nt)}{N}, \ldots,\frac{L_{J}^N(Nt)}{N}\right)
\\= (\gamma_0(t),0^{(\kappa)},\gamma_{\kappa+1}(t),\ldots,\gamma_J(t)),
\end{multline*}
where $0^{(\kappa)}$ is the $\kappa$-th dimensional zero vector;\medskip
\item If $\beta_\kappa^*{>}1$,  
\begin{multline*}
\lim_{N\rightarrow \infty}\left(\!\frac{L_{0}^N(Nt)}{N}, \frac{L_1^N(Nt)}{(\log N)^3}, \ldots, \frac{L_{\kappa-1}^N(Nt)}{(\log N)^3}, \frac{L_{\kappa}^N(Nt)}{N^{\alpha_\kappa^*{-}(J{-}\kappa)}}, \frac{L_{\kappa+1}^N(Nt)}{N}, \ldots,\frac{L_{J}^N(Nt)}{N}\!\right)
\\= \left(\gamma_0(t),0^{(\kappa-1)},\frac{1}{{\gamma}_{\kappa+1}(t){\gamma}_{\kappa+2}(t)\cdots {\gamma}_J(t)},{\gamma}_{\kappa+1}(t),\ldots,{\gamma}_J(t)\right),
\end{multline*}
    \end{enumerate}
\item If $\kappa=J$, 
  \[
\lim_{N\rightarrow \infty}\bigg(\frac{L_0^N(Nt)}{N}, \frac{L_1^N(Nt)}{(\log N)^3}, \ldots, \frac{L_{J-1}^N(Nt)}{(\log N)^3}, \frac{L_J^N(Nt)}{N^{\alpetu{J}}} \bigg) = \big({\gamma}_0(t), 0^{(J{-}1)}, {\gamma}_0(t)^{\alpetu{J}}\big).
\]
  \end{enumerate}
\end{theorem}
\noindent
Note that,  by definition of $\kappa$, we have
\[
\frac{\alpetu{\kappa}}{J{-}\kappa {+}1}<1\leq \frac{\alpetu{\kappa{+}1}}{J{-}\kappa}\,\text{ and }
\beta_\kappa^*\in \left(\frac{\alpetu{\kappa}}{J{-}\kappa {+}1},\frac{\alpetu{\kappa{+}1}}{J{-}\kappa}\right).
 \]
 Cases 2a) or 2b) depend on $\beta_\kappa^*$ being before or after $1$ in the last time interval. Either the  queue with index $\kappa$ has the time to come back to $0$ on the time scale $(N^t, t{\in}(0,1))$, corresponding to case 2)a), or it does not, and this is case 2)b).
All  other results are direct consequences of Theorems~\ref{th: fluid limit} and~\ref{th: fluid general}.

\providecommand{\bysame}{\leavevmode\hbox to3em{\hrulefill}\thinspace}
\providecommand{\MR}{\relax\ifhmode\unskip\space\fi MR }
\providecommand{\MRhref}[2]{%
  \href{http://www.ams.org/mathscinet-getitem?mr=#1}{#2}
}
\providecommand{\href}[2]{#2}

\end{document}